\documentclass[12pt]{article}

\usepackage{amsthm,amssymb, amsmath, amscd, amsfonts,graphicx, latexsym,verbatim}
\usepackage{amsmath}
\usepackage{cite, psfrag,color}
\usepackage{titlesec}
\usepackage{epstopdf}
\usepackage{fancybox}
\usepackage{float}
\usepackage[colorlinks,citecolor=blue]{hyperref}
\usepackage[numbers,square]{natbib}
\usepackage[font=normalsize,labelsep=none]{caption}
\usepackage{diagbox}
\usepackage{changepage}
\usepackage{enumerate}
\usepackage{arydshln}
\usepackage{bbm}
\usepackage{setspace}
\usepackage{titletoc}
\usepackage{mdframed}
\usepackage{algorithm,algorithmic}
\usepackage{empheq}   
\usepackage{framed}
\usepackage{hyperref}

 \newtheorem{theorem}{Theorem}[section]
 
 \newtheorem{lemma}{Lemma}[section]
 \newtheorem{corollary}{Corollary}[section]
 \newtheorem{example}{Example}[section]
 \newtheorem{remark}{Remark}[section]

\usepackage{enumitem}
\setenumerate[1]{itemsep=0pt,partopsep=0pt,parsep=\parskip,topsep=2pt}
\setitemize[1]{itemsep=0pt,partopsep=0pt,parsep=\parskip,topsep=2pt}
\setdescription{itemsep=0pt,partopsep=0pt,parsep=\parskip,topsep=2pt}

\usepackage{mdframed}
\mdfsetup{
	linewidth=0.75pt,        
}

\titleformat{\section} 
{\raggedright\large\bfseries\bf} 
{\thesection .\quad} 
{0pt} 
{}  
 
\titleformat{\subsection} 
{\raggedright\normalsize\bfseries\it} 
{\thesubsection .\quad} 
{0pt} 
{} 


\newlength{\boxedparwidth}
  {\begin{center} \begin{tabular}{|@{\hspace{.315in}}c@{\hspace{.15in}}|}
                  \hline \\ \begin{minipage}[t]{\boxedparwidth}
                  \setlength{\parindent}{.25in}}%
  {\end{minipage} \\ \\ \hline \end{tabular} \end{center}}

\begin{document}
\captionsetup[figure]{labelfont={bf},labelformat={default},labelsep=period,name={Fig.},font={footnotesize}}
\captionsetup[table]{labelfont={bf},labelformat={default},labelsep=period,name={Table},font={footnotesize}}

\begin{center}
{\Large{\bf{New $r$-Euler--Mahonian statistics involving Denert's statistic
			\vskip 1mm}}}
\end{center}
\vspace{0mm}
\begin{center}
	\text{Shao-Hua Liu}\\
	\vskip 1mm
	{School of Statistics and Data Science\\
		Guangdong University of Finance and Economics\\
		Guangzhou, China\\
		Email: liushaohua@gdufe.edu.cn}
\end{center}
 \title{}

 \vskip 1mm

\noindent {\bf Abstract.}
Recently, we proved the equidistribution of the pairs of permutation statistics $(r\textsf{des},r\textsf{maj})$ and $(r\textsf{exc},r\textsf{den})$.
Any pair of permutation statistics that is equidistributed with these pairs is said to be $r$-Euler--Mahonian.
Several classes of $r$-Euler--Mahonian statistics were established by Huang--Lin--Yan and Huang--Yan.
Inspired by their bijections,
we provide a new bijective proof of the classical result that $(\textsf{exc},\textsf{den})$ is Euler--Mahonian.
Using this bijection, we further show that $(\textsf{exc}_{r},\textsf{den})$ is $r$-Euler--Mahonian,
where $\textsf{exc}_{r}$ denotes the number of $r$-level excedances (i.e., excedances at least $r$).
Furthermore, by extending our bijection, 
we establish a more general result that encompasses all the aforementioned results.

   \vskip -2mm

\titlecontents{section}[1.5em]
{\footnotesize \vspace{-10pt}}
{\contentslabel{1.5em}}{\hspace*{-1.6em}}
{~\titlerule*[0.6pc]{$.$}~\contentspage}

\titlecontents{subsection}[3.5em]
{\footnotesize \vspace{-10pt}}
{\contentslabel{2.3em}}{\hspace*{-4em}}
{~\titlerule*[0.6pc]{$.$}~\contentspage}


\section{Introduction}
\subsection{Euler--Mahonian statistics} 
Let $\mathfrak{S}_{n}$ denote the symmetric group of permutations on $[n]:=\{1,2,\ldots,n\}$.
We treat permutations as words on distinct letters by writing them in one-line notation; 
that is, if  $\sigma\in\mathfrak{S}_{n}$,
then we write $\sigma=\sigma_{1}\sigma_{2}\ldots\sigma_{n}$ where 
$\sigma_{i}=\sigma(i)$ for all $i\in[n]$.
Given a permutation $\sigma=\sigma_{1}\sigma_{2}\ldots\sigma_{n}\in\mathfrak{S}_{n}$,
a position $i$, $1\leq i\leq n-1$, is called a \emph{descent} of $\sigma$ if $\sigma_{i}>\sigma_{i+1}$.
Let $\textsf{Des}(\sigma)$ be the set  of descents of $\sigma$,
and let $\textsf{des}(\sigma)=|\textsf{Des}(\sigma)|$,
where $|\cdot|$ indicates cardinality.
A position $i$, $1\leq i\leq n$, is called 
an \emph{excedance} of $\sigma$ if $\sigma_{i}>i$.
Let $\textsf{Exc}(\sigma)$ be the set of excedances of $\sigma$,
and let $\textsf{exc}(\sigma)=|\textsf{Exc}(\sigma)|$.
It is well known that $\textsf{des}$ and $\textsf{exc}$ are equidistributed over $\mathfrak{S}_{n}$,
and a permutation statistic equidistributed with them is said to be \emph{Eulerian}.

Let $\sigma=\sigma_{1}\sigma_{2}\ldots\sigma_{n}\in\mathfrak{S}_{n}$.
A pair $(i,j)$ of positions  is called an \emph{inversion} of $\sigma$ if $i<j$ and $\sigma_{i}>\sigma_{j}$. 
Let $\textsf{Inv}(\sigma)$ be the set of inversions of $\sigma$,
and let $\textsf{inv}(\sigma)=|\textsf{Inv}(\sigma)|$.
Define the \emph{major index} of $\sigma$, denoted by $\textsf{maj}(\sigma)$,  to be  the sum of the descents of $\sigma$, that is,
$$\textsf{maj}(\sigma)=\sum_{i\in\textsf{Des}(\sigma)}i.$$
In \cite{MacMahon-1916},  MacMahon showed that \textsf{inv} and \textsf{maj} are equidistributed over $\mathfrak{S}_{n}$, and that
\begin{align*} 
	\sum_{\sigma\in\mathfrak{S}_{n}}q^{\textsf{inv}(\sigma)}=\sum_{\sigma\in\mathfrak{S}_{n}}q^{\textsf{maj}(\sigma)}=[n]_{q}!,
\end{align*}
where  
$[n]_{q}!=[n]_{q}[n-1]_{q}\ldots[1]_{q}$
with 
$[n]_{q}=1+q+\cdots+q^{n-1}.$
In his honor, any permutation statistic with this distribution over $\mathfrak{S}_{n}$ is said to be \emph{Mahonian}.

Denert \cite{Denert-1990} introduced an excedance-based Mahonian statistic;
in this paper, we adopt an equivalent definition due to  Foata and Zeilberger \cite{Foata-1990}.
Let $\sigma=\sigma_{1}\sigma_{2}\ldots \sigma_{n}\in\mathfrak{S}_{n}$.
If $i$ is an excedance,
we call $\sigma_{i}$ an \emph{excedance-letter}.
Let $\textsf{\scriptsize{EXCL}}(\sigma)$
be the subsequence of $\sigma$ that consists of the excedance-letters,
and let $\textsf{\scriptsize{NEXCL}}(\sigma)$ be the subsequence of $\sigma$  that consists of the remaining letters (i.e., the non-excedance-letters).
Define \emph{Denert's statistic} of $\sigma$, denoted by $\textsf{den}(\sigma)$, as
$$\textsf{den}(\sigma)=\sum_{i\in\textsf{Exc}(\sigma)}i+\textsf{inv}(\textsf{\scriptsize{EXCL}}(\sigma))+
\textsf{inv}(\textsf{\scriptsize{NEXCL}}(\sigma)).$$
For example, if  $\sigma=715492638$, 
we have 
$\textsf{den}(\sigma)=1+3+5+\textsf{inv}(759)+\textsf{inv}(142638)=13$.

A pair of permutation statistics that is equidistributed with $(\textsf{des},\textsf{maj})$ is said to be \emph{Euler--Mahonian}.
Denert \cite{Denert-1990} conjectured that the pair $(\textsf{exc},\textsf{den})$ is Euler--Mahonian.
The first proof of Denert's conjecture was given by Foata and Zeilberger \cite{Foata-1990},
and a direct bijective proof was later provided by Han \cite{Han-1990-direct}.
In Section \ref{new proof of Denert's conjecture}, we present another bijective proof.
\subsection{$r$-Euler--Mahonian statistics} 
Throughout the paper, let $r\geq1$ be an integer.
Let $\sigma=\sigma_{1}\sigma_{2}\ldots\sigma_{n}\in\mathfrak{S}_{n}$. 
A position $i$, $1\leq i\leq n-1$, is called an \emph{$r$-descent} (or \emph{$r$-gap descent}) of $\sigma$
if $\sigma_{i}\geq\sigma_{i+1}+r$.
Let $r\textsf{Des}(\sigma)$ denote the set  of $r$-descents of $\sigma$,
and let $r\textsf{des}(\sigma)=|r\textsf{Des}(\sigma)|$.
A position $i$, $1\leq i\leq n$, is called 
an \emph{$r$-excedance} (or \emph{$r$-gap excedance}) of $\sigma$ if $\sigma_{i}\geq i+r$.
Let $r\textsf{Exc}(\sigma)$ denote the set of $r$-excedances of $\sigma$,
and let $r\textsf{exc}(\sigma)=|r\textsf{Exc}(\sigma)|$.
Clearly, when $r=1$, $r\textsf{des}$ and $r\textsf{exc}$ reduce to \textsf{des} and \textsf{exc}, respectively.
The statistics $r\textsf{des}$ and $r\textsf{exc}$ are equidistributed over $\mathfrak{S}_{n}$,
and any permutation statistic equidistributed with them is said to be $r$-\emph{Eulerian}.

We now review two interpolating Mahonian statistics:
the $r$-major index, introduced by Rawlings \cite{Rawlings-1981},
and the $r$-Denert's statistic, introduced by Han \cite{Han-1991-thesis}.

Let $\sigma=\sigma_{1}\sigma_{2}\ldots \sigma_{n}\in\mathfrak{S}_{n}$.
The \emph{$r$-major index} of $\sigma$, 
denoted by $r\textsf{maj}(\sigma)$, is defined as
\begin{align*}
	r\textsf{maj}(\sigma)=\sum_{i\in r\textsf{Des}(\sigma)}i +|r\textsf{Inv}(\sigma)|,
\end{align*}
where $r\textsf{Inv}(\sigma)=\{(i,j)\in\textsf{Inv}(\sigma): \sigma_{i}<\sigma_{j}+r\}.$
Clearly, $r\textsf{maj}$ reduces to $\textsf{maj}$  when $r=1$ and to $\textsf{inv}$ when $r\geq n$,
thus the family of $r\textsf{maj}$ interpolates between \textsf{maj} and \textsf{inv}.

Let $\sigma=\sigma_{1}\sigma_{2}\ldots \sigma_{n}\in\mathfrak{S}_{n}$. 
If $i$ is an $r$-excedance,
we call $\sigma_{i}$ an \emph{$r$-excedance-letter}.
Let $r\textsf{\scriptsize{EXCL}}(\sigma)$ be the subsequence of $\sigma$ that consists of the $r$-excedance-letters,
and let $r\textsf{\scriptsize{NEXC}}(\sigma)$ be the subsequence of $\sigma$ that consists of the remaining letters (i.e., the non-$r$-excedance-letters).
Define the \emph{$r$-Denert's statistic} of $\sigma$, denoted by $r\textsf{den}(\sigma)$, to be 
\begin{align*}
	r\textsf{den}(\sigma)=\sum_{i\in r\textsf{Exc}(\sigma)}(i+r-1)+\textsf{inv}(r\textsf{\scriptsize{EXCL}}(\sigma))+\textsf{inv}(r\textsf{\scriptsize{NEXCL}}(\sigma)).
\end{align*}
Clearly, $r\textsf{den}$ reduces to $\textsf{den}$  when $r=1$ and to $\textsf{inv}$ when $r\geq n$.

In \cite{Liu-2024}, we proved that $(r\textsf{des},r\textsf{maj})$ and $(r\textsf{exc},r\textsf{den})$ are equidistributed over $\mathfrak{S}_{n}$.
Any pair of permutation statistics equidistributed with these pairs is said to be \emph{$r$-Euler--Mahonian}.
In the remainder of this subsection, we recall other $r$-Euler--Mahonian statistics.

Let $\sigma=\sigma_{1}\sigma_{2}\ldots \sigma_{n}\in\mathfrak{S}_{n}$. 
An excedance $i$ of $\sigma$ is called an \emph{$r$-level excedance} of $\sigma$ if  $i\geq r$.
Let $\textsf{Exc}_{r}(\sigma)$ and $\textsf{exc}_{r}(\sigma)$ denote the set and the number of $r$-level excedances of $\sigma$, respectively. That is,
$$\textsf{Exc}_{r}(\sigma)=\{i\in[n]:\sigma_{i}>i,~i\geq r\}\text{~~and~~}\textsf{exc}_{r}(\sigma)=|\textsf{Exc}_{r}(\sigma)|.$$
For example, if $\sigma=2715643$, we have  $\textsf{exc}_{1}=\textsf{exc}=|\{1,2,4,5\}|=4$,
$\textsf{exc}_{2}(\sigma)=|\{2,4,5\}|=3$,
$\textsf{exc}_{3}(\sigma)=\textsf{exc}_{4}(\sigma)=|\{4,5\}|=2$,
$\textsf{exc}_{5}(\sigma)=|\{5\}|=1$,
$\textsf{exc}_{6}(\sigma)=0$.

Given a permutation  $\sigma=\sigma_{1}\sigma_{2}\ldots \sigma_{n}\in\mathfrak{S}_{n}$, 
an excedance-letter $\sigma_{i}$ is called an \emph{$r$-level excedance-letter} of $\sigma$ if $\sigma_{i}\geq r$.
The position $i$ is then referred to as an \emph{$r$-level excedance-letter position}.
Let $\textsf{Exclp}_{r}(\sigma)$ denote the set of $r$-level excedance-letter positions of $\sigma$, that is,
$$\textsf{Exclp}_{r}(\sigma)=\{i\in[n]:\sigma_{i}>i,~\sigma_{i}\geq r\}.$$
Let $\textsf{\scriptsize{EXCL}}_{r}(\sigma)$ be the subsequence of $\sigma$  that consists of the $r$-level excedance-letters,
and let $\textsf{\scriptsize{NEXCL}}_{r}(\sigma)$ be the subsequence of $\sigma$  that consists of the remaining letters 
(i.e., the non-$r$-level excedance-letters). 
Define the \emph{$r$-level Denert's statistic} of $\sigma$, denoted by $\textsf{den}_{r}(\sigma)$, to be 
$$
	\textsf{den}_{r}(\sigma)=\sum_{i\in\textsf{Exclp}_{r}(\sigma)}i+\textsf{inv}(\textsf{\scriptsize{EXCL}}_{r}(\sigma))+\textsf{inv}(\textsf{\scriptsize{NEXCL}}_{r}(\sigma)).
$$
For example, let $\sigma=2715643$, we have $\textsf{\scriptsize{EXCL}}(\sigma)=2756$.
If $r=3$, we have $\textsf{\scriptsize{EXCL}}_{3}(\sigma)=756$ and $\textsf{Exclp}_{3}(\sigma)=\{2,4,5\}$, then 
$\textsf{den}_{3}(\sigma)=2+4+5+\textsf{inv}(756)+\textsf{inv}(2143)=15.$
If $r=6$, we have $\textsf{\scriptsize{EXCL}}_{6}(\sigma)=76$ and $\textsf{Exclp}_{6}(\sigma)=\{2,5\}$, then 
$\textsf{den}_{6}(\sigma)=2+5+\textsf{inv}(76)+\textsf{inv}(21543)=12.$

Note that if $i$ is an $r$-level excedance of $\sigma$,
then $i$ is an $r$-level excedance-letter position of $\sigma$,
but not vice versa.
Therefore, 
$\textsf{Exc}_{r}(\sigma)\subseteq\textsf{Exclp}_{r}(\sigma)
\text{~~and~~} \textsf{exc}_{r}(\sigma)\leq|\textsf{Exclp}_{r}(\sigma)|.$

Huang, Lin and Yan \cite{Yan-2025} proved that the pair $(\textsf{exc}_{r},\textsf{den}_{r})$ is $r$-Euler--Mahonian,
thereby confirming a conjecture proposed in \cite{Liu-2024}.
\begin{theorem}[Huang, Lin and Yan \cite{Yan-2025}] 
The pair $(\emph{\textsf{exc}}_{r},\emph{\textsf{den}}_{r})$ is $r$-Euler--Mahonian.
\end{theorem}

Let us generalize the above notions (see \cite{Liu-2024}, where slightly different notations are used).
Let $g,\ell\geq1$, which we assume throughout the paper.
Given  a permutation  $\sigma=\sigma_{1}\sigma_{2}\ldots \sigma_{n}\in\mathfrak{S}_{n}$,
a position $i$, $1\leq i\leq n$, is called a  \emph{$g$-gap $\ell$-level excedance} of $\sigma$ if $\sigma_{i}\geq i+g$ and $i\geq \ell$.
Let $g\textsf{Exc}_{\ell}(\sigma)$ and $g\textsf{exc}_{\ell}(\sigma)$ denote  the set and the number of $g$-gap $\ell$-level excedances of $\sigma$, respectively.
That is,
\vspace{-8pt}
$$g\textsf{Exc}_{\ell}(\sigma)=\{i\in[n]:\sigma_{i}\geq i+g ,~i\geq \ell\}
\text{~~and~~}g\textsf{exc}_{\ell}(\sigma)=|g\textsf{Exc}_{\ell}(\sigma)|.$$

\vspace{-8pt}
Let $\sigma=\sigma_{1}\sigma_{2}\ldots \sigma_{n}\in\mathfrak{S}_{n}$.
The letter $\sigma_{i}$ is called a \emph{$g$-gap $\ell$-level excedance-letter} of $\sigma$ if $\sigma_{i}\geq i+g$ and $\sigma_{i}\geq \ell$.
The position $i$ is then referred to as a \emph{$g$-gap $\ell$-level excedance-letter position} of $\sigma$.
Let $g\textsf{Exclp}_{\ell}(\sigma)$ be the set of $g$-gap $\ell$-level  excedance-letter positions of $\sigma$; i.e.,
$$g\textsf{Exclp}_{\ell}(\sigma)=\{i\in[n]:\sigma_{i}\geq i+g,~ \sigma_{i}\geq \ell\}.$$
Let $g\textsf{\scriptsize{EXCL}}_{\ell}(\sigma)$ be the subsequence of $\sigma$  that consists of the $g$-gap $\ell$-level excedance-letters,
and let $g\textsf{\scriptsize{NEXC}}_{\ell}(\sigma)$ be the subsequence of $\sigma$  that consists of the remaining letters
(i.e., the non-$g$-gap $\ell$-level excedance-letters). 
Define the \emph{$g$-gap $\ell$-level Denert's statistic} of $\sigma$, denoted by $g\textsf{den}_{\ell}(\sigma)$, as 
\vspace{-8pt}
$$
g\textsf{den}_{\ell}(\sigma)=\sum_{i\in g\textsf{Exclp}_{\ell}(\sigma)}(i+g-1)+\textsf{inv}(g\textsf{\scriptsize{EXCL}}_{\ell}(\sigma))+\textsf{inv}(g\textsf{\scriptsize{NEXCL}}_{\ell}(\sigma)).
$$

\vspace{-8pt}
Huang and Yan \cite{Yan-2025-2} proved that  $(g\textsf{exc}_{\ell},g\textsf{den}_{\ell})$ is $(g+\ell-1)$-Euler--Mahonian, thereby confirming a conjecture proposed in \cite{Liu-2024}.
(By setting $\ell=1$ and $g=1$ in this result,
one obtains that $(r\textsf{exc},r\textsf{den})$ and $(\textsf{exc}_{r},\textsf{den}_{r})$ are $r$-Euler--Mahonian,
respectively.)
Moreover, they obtained an instructive result showing that $(g\textsf{exc}_{\ell},g\textsf{den}_{g+\ell})$ is also $(g+\ell-1)$-Euler--Mahonian.	
\begin{theorem}[Huang and Yan \cite{Yan-2025-2}]\label{thm-Huang-Yan} 
	Let $g,\ell\geq1$. We have
	
	(1) The pair $(g\emph{\textsf{exc}}_{\ell},g\emph{\textsf{den}}_{\ell})$ is $(g+\ell-1)$-Euler--Mahonian.
	
	(2) The pair $(g\emph{\textsf{exc}}_{\ell},g\emph{\textsf{den}}_{g+\ell})$ is $(g+\ell-1)$-Euler--Mahonian.	
\end{theorem}

\subsection{The main results}
In the spirit of the map $\gamma_{n,g,\ell}$ of Huang and Yan \cite{Yan-2025-2} and the map $\beta_{r,n}$ of Huang, Lin and Yan \cite{Yan-2025},
we provide a new bijective proof of the classical result that $(\textsf{exc},\textsf{den})$ is Euler--Mahonian.
Using this bijection, we further prove that
$(\textsf{exc}_{r},\textsf{den})$ is $r$-Euler--Mahonian,
which is somewhat surprising since  the second statistic, \textsf{den}, is independent of $r$.

\begin{theorem}\label{mian-theorem-(excr,den)}
Let $r\geq1$. The pair  $(\emph{\textsf{exc}}_{r},\emph{\textsf{den}})$ is $r$-Euler--Mahonian.
\end{theorem}
\begin{remark}
	\emph{We note that the result does \emph{not} hold for the ``gap'' case---that is, $(r\textsf{exc},\textsf{den})$ is \emph{not}  $r$-Euler--Mahonian. }
\end{remark}
\newpage
By extending our bijection, we establish the following more general result.
\begin{theorem}\label{mian-general-theorem}
	Let $g,\ell\geq1$.
	The pair $(g\emph{\textsf{exc}}_{\ell},g\emph{\textsf{den}}_{h})$ is $(g+\ell-1)$-Euler--Mahonian for all 
	$1\leq h\leq g+\ell$.
\end{theorem}
\begin{remark}
	\emph{Small examples show that $(g\textsf{exc}_{\ell},g\textsf{den}_{h})$ is \emph{not}  
	$(g+\ell-1)$-Euler--Mahonian when $h>g+\ell$. 
	Thus, the bound $g+\ell$ indicated by Theorem \ref{thm-Huang-Yan} (2) is tight.
	}
\end{remark} 
Setting $g=h=1$ in Theorem \ref{mian-general-theorem} yields Theorem  \ref{mian-theorem-(excr,den)}.
Setting $h=\ell$ and $h=g+\ell$ 
recovers  Theorem \ref{thm-Huang-Yan} (1) and (2), respectively. 
Finally, setting $g=1$ yields the following result.
\begin{corollary} \label{corr-1}
	Let $r\geq1$.
	The pair $(\emph{\textsf{exc}}_{r},\emph{\textsf{den}}_{h})$ is $r$-Euler--Mahonian for all $1\leq h\leq r+1$.
\end{corollary}
This paper is organized as follows. 
Section \ref{section-new-proof} presents a new bijective proof of Denert's conjecture. 
Section \ref{Proof of Theorem 1.3} proves Theorem \ref{mian-theorem-(excr,den)} using this bijection, 
and Section \ref{Proof of Theorem 1.4} extends the bijection to prove Theorem \ref{mian-general-theorem}.

\section{A new proof of Denert's conjecture}\label{section-new-proof}
\label{new proof of Denert's conjecture}
In this section, we present a bijective proof of Denert's conjecture,  
which states that $(\textsf{exc},\textsf{den})$ is Euler--Mahonian.
To prove that a pair $(\textsf{stat}_{1},\textsf{stat}_{2})$ of permutation statistics is Euler--Mahonian, 
it suffices to show that there exists a bijection
\vspace{-6pt}
 $$\phi_{n}:\mathfrak{S}_{n-1}\times\{0,1,\ldots,n-1\}\rightarrow\mathfrak{S}_{n}$$
 
\vspace{-10pt}
\noindent such that for $\sigma\in\mathfrak{S}_{n-1}$ and $0\leq c\leq n-1$, 
we have
\begin{equation}\label{eq-Euler--Mahonian}
	\begin{split}
		\textsf{stat}_{1}(\phi_{n}(\sigma,c))&=\left
		\{
		\begin{aligned}
			&\textsf{stat}_{1}(\sigma), \quad\quad\quad\text{if~}0\leq c\leq \textsf{stat}_{1}(\sigma),\\
			&\textsf{stat}_{1}(\sigma)+1,\quad\text{~otherwise,}
		\end{aligned}
		\right.\\				
		\textsf{stat}_{2}(\phi_{n}(\sigma,c))&=\textsf{stat}_{2}(\sigma)+c, 
	\end{split}
\end{equation}
see \cite{Rawlings-1981,Han-1990-direct}. 
The main result of this section is the following theorem.
\begin{theorem}\label{Thm-New-bijection}
	There is a bijection $$\phi_{n}^{\emph{\text{den}}}: \mathfrak{S}_{n-1}\times\{0,1,\ldots,n-1\}\rightarrow\mathfrak{S}_{n},$$
	such that for $\sigma\in\mathfrak{S}_{n-1}$ and $0\leq c\leq n-1$, 
	we have 
\begin{equation}\label{eq-(exc,den)-Exc-den}
	\begin{split}
	\emph{\textsf{Exc}}(\phi_{n}^{\emph{\text{den}}}(\sigma,c))&=
	\left\{
	\begin{aligned}
		&\emph{\textsf{Exc}}(\sigma), \quad\quad\quad\quad~\text{if~}0\leq c\leq \emph{\textsf{exc}}(\sigma),\\
		&\emph{\textsf{Exc}}(\sigma)\cup\{k_{d}\},\quad\text{~otherwise}
	\end{aligned}
	\right.\\
	\emph{\textsf{den}}(\phi_{n}^{\emph{\text{den}}}(\sigma,c))&=\emph{\textsf{den}}(\sigma)+c,
		\end{split}
\end{equation}
where $d=c-\emph{\textsf{exc}}(\sigma)$ and
$k_{d}$ is the $d$-th non-excedance of $\sigma$ from smallest to largest.
\end{theorem}
It is clear that the top equation in (\ref{eq-(exc,den)-Exc-den}) is a strengthening of the top equation in (\ref{eq-Euler--Mahonian}).
Therefore, the bijection $\phi_{n}^{\text{den}}$ in Theorem \ref{Thm-New-bijection} yields a bijective proof of Denert's conjecture. 
We now present this bijection.

\vskip 2pt	

\begin{framed}
\begin{center}
\vskip -6pt
{\bf{The map}} $\phi_{n}^{\text{den}}: \mathfrak{S}_{n-1}\times\{0,1,\ldots,n-1\}\rightarrow\mathfrak{S}_{n}$
\end{center}
\vskip -6pt

Let $\sigma=\sigma_{1}\sigma_{2}\ldots\sigma_{n-1}\in\mathfrak{S}_{n-1}$ and let $c$ be an integer with $0\leq c\leq n-1$.
Let $\textsf{exc}(\sigma)=s$, and let the excedance-letters of $\sigma$ be
$e_{1},e_{2},\ldots,e_{s}$, where $e_{1}<e_{2}<\cdots<e_{s}$.
Let the non-excedance-letters of $\sigma$ be 
$\sigma_{k_{1}},\sigma_{k_{2}},\ldots,\sigma_{k_{t}}$ (equivalently, the non-excedances of $\sigma$ are $k_{1},k_{2},\ldots,k_{t}$), where $k_{1}<k_{2}<\cdots<k_{t}$. 
We distinguish three cases.
\begin{itemize}[leftmargin=1.5em]
\vskip 2pt	
\item[$\bullet$]
 Case 1: $c=0$. 
 Define $\phi_{n}^{\text{den}}(\sigma,c)=\sigma_{1}\sigma_{2}\ldots\sigma_{n-1}n$.

\vskip 2pt	
\item[$\bullet$] 
Case 2: $1\leq c\leq s$. 
Let $d=s+1-c$.
Assume that $e_{j_{1}}<e_{j_{2}}<\cdots<e_{j_{x}}$ are the excedance-letters of $\sigma$ that lie strictly to the left of $\sigma_{e_{d}}$ and are at least $e_{d}$.
Clearly, $e_{j_{1}}=e_{d}$.
Assume that $\sigma_{k_{y}},\sigma_{k_{y+1}},\ldots,\sigma_{k_{t}}$ are the non-excedance-letters of $\sigma$ that lie weakly to the right of $\sigma_{e_{d}}$.
Define $\phi_{n}^{\text{den}}(\sigma,c)$  as the permutation obtained from $\sigma$ by the following three steps:

\vskip 2pt	

\textbf{Step 1.  Adjusting Excedance-Letters}\\
	Replace $e_{j_{i}}$ with $e_{j_{i+1}}$ for  $i=1,2,\ldots,x$, with the convention that $e_{j_{x+1}}=n$.
	
\vskip 2pt	
	
\textbf{Step 2.  Shifting Non-Excedance-Letters}\\
	Replace  $\sigma_{k_{i}}$ with $\sigma_{k_{i-1}}$ for $i=y+1,y+2,\ldots,t,t+1$, with the convention that $k_{t+1}=n$
	(in other words, shift the letters $\sigma_{k_{y}},\sigma_{k_{y+1}},\ldots,\sigma_{k_{t}}$ one non-excedance  to the right).
	
\vskip 2pt		
	
\textbf{Step 3.  Placing $e_{d}$}\\
	Set the $k_{y}$-th letter to be $e_{d}$.

(See Example \ref{example-Case2} for an instance in Case 2.)

\vskip 2pt	
\item[$\bullet$] Case 3: $s+1\leq c\leq n-1$.
Let $d=c-s$. 
Define $\phi_{n}^{\text{den}}(\sigma,c)$ as the permutation obtained from $\sigma$ by the following two steps:

\vskip 2pt		

\textbf{Step i. Shifting Non-Excedance-Letters}\\
Replace $\sigma_{k_{i}}$ with $\sigma_{k_{i-1}}$ for $i=d+1,d+2,\ldots,t,t+1$, with the convention that $k_{t+1}=n$
(in other words, shift the letters $\sigma_{k_{d}},\sigma_{k_{d+1}},\ldots,\sigma_{k_{t}}$
one non-excedance to the right).

\vskip 2pt		

\textbf{Step ii. Placing $n$} \\
Set the $k_{d}$-th letter to be $n$.    

(See Example \ref{example-Case3} for an instance in Case 3.)
\end{itemize}
\end{framed}

\begin{figure}[t]
	\centering
	\includegraphics[width=15.2cm]{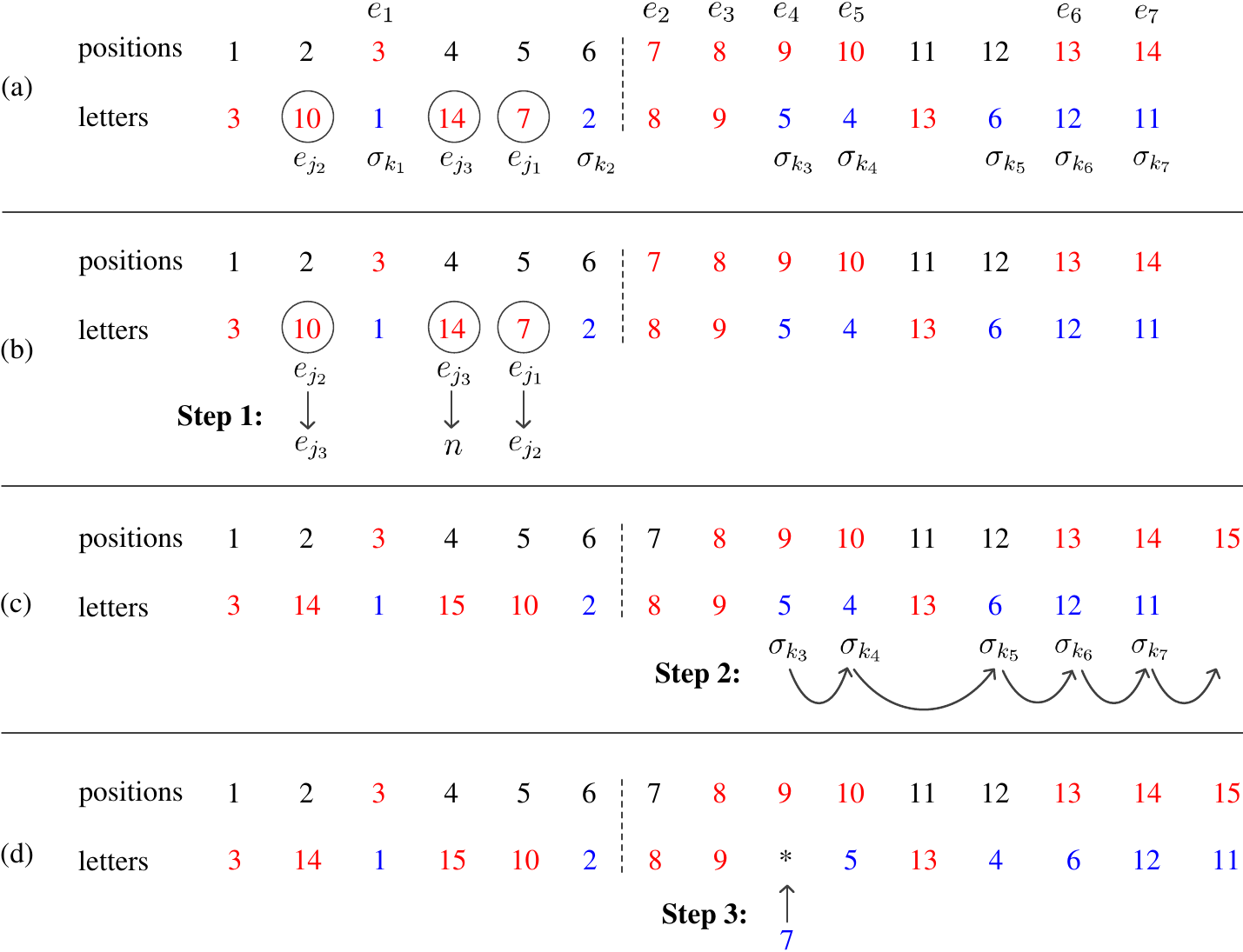}
	\caption{~ An illustration for Example \ref{example-Case2}.
	}\label{Fig_exm_1}
\end{figure}
\begin{example}\label{example-Case2}
\emph{
Let $\sigma=3~10~1~14~7~2~8~9~5~4~13~6~12~11$, and let $c=6$.
Clearly, $s=\textsf{exc}(\sigma)=7$,
and the excedance-letters of $\sigma$ are $e_{1}=3,~e_{2}=7,~e_{3}=8,~e_{4}=9,~e_{5}=10,~e_{6}=13,~e_{7}=14.$
The non-excedance-letters are
$\sigma_{k_{1}}=1,\sigma_{k_{2}}=2,\sigma_{k_{3}}=5,\sigma_{k_{4}}=4,\sigma_{k_{5}}=6,\sigma_{k_{6}}=12,\sigma_{k_{7}}=11.$
Note that $1\leq c\leq s$.
We have $d=s+1-c=2$, so $e_{d}=e_{2}=7$.
The excedance-letters of $\sigma$ that lie strictly to the left of $\sigma_{e_{d}}=\sigma_{7}$ and that are at least $e_{d}=7$ are
$e_{j_{1}}=7,~e_{j_{2}}=10,~e_{j_{3}}=14.$
The non-excedance-letters of $\sigma$ that lie weakly to the right of $\sigma_{e_{d}}=\sigma_{7}$ are $\sigma_{k_{3}}=5,~\sigma_{k_{4}}=4,~\sigma_{k_{5}}=6,~\sigma_{k_{6}}=12,~\sigma_{k_{7}}=11.$
See Fig. \ref{Fig_exm_1}(a) for an illustration, where
the positions and letters $e_{1},e_{2},\ldots,e_{7}$ are in red,
the non-excedance-letters are in blue, 
and the letters $e_{j_{1}},e_{j_{2}},e_{j_{3}}$ are circled.
\\
\textbf{Step 1. Adjusting Excedance-Letters}\\
Replace $e_{j_{1}}=7$ with $e_{j_{2}}=10$,
$e_{j_{2}}=10$ with $e_{j_{3}}=14$,
and $e_{j_{3}}=14$ with $e_{j_{4}}=n=15$;
see Fig. \ref{Fig_exm_1}(b) for an illustration.
This results in the sequence: $3~14~1~15~10~2~8~9~5~4~13~6~12~11$.
\\
\textbf{Step 2. Shifting Non-Excedance-Letters}\\
Shift  $\sigma_{k_{3}},\sigma_{k_{4}},\ldots,\sigma_{k_{7}}$
one non-excedance  to the right;
see Fig. \ref{Fig_exm_1}(c) for an illustration.
This results in the sequence: $3~14~1~15~10~2~8~9*5~13~4~6~12~11$,
where $\ast$ is the $k_{3}$-th letter.
\\
\textbf{Step 3. Placing $e_{d}$}\\
Set the $k_{3}$-th letter to be $e_{d}=7$;
see Fig. \ref{Fig_exm_1}(d) for an illustration.
Then we obtain $\phi_{15}^{\text{den}}(\sigma,6)=3~14~1~15~10~2~8~9~7~5~13~4~6~12~11$.
}\end{example}
\begin{figure}[t]
	\centering
	\includegraphics[width=15.2cm]{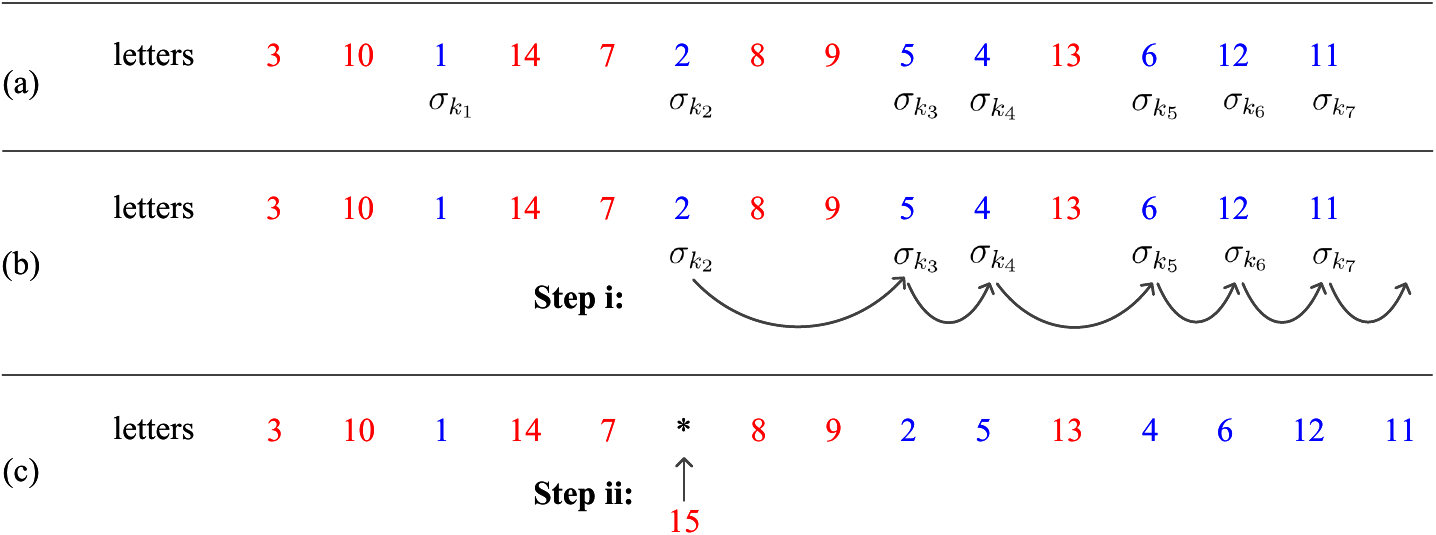}
	\caption{~ An illustration for Example \ref{example-Case3}.
	}\label{Fig_exm_2}
\end{figure}
\begin{example}\label{example-Case3}
\emph{
Consider the running example
$\sigma=3~10~1~14~7~2~8~9~5~4~13~6~12~11$, and let $c=9$.
Note that $c\geq s+1=8$. 
In this case, $d=c-s=9-7=2$.
Recall that the non-excedance-letters are
$\sigma_{k_{1}}=1,\sigma_{k_{2}}=2,\sigma_{k_{3}}=5,\sigma_{k_{4}}=4,\sigma_{k_{5}}=6,\sigma_{k_{6}}=12,\sigma_{k_{7}}=11.$
See Fig. \ref{Fig_exm_2}(a) for an illustration,
where the excedance-letters and non-excedance-letters of $\sigma$ are in red and blue, respectively.
\\
\textbf{Step i. Shifting Non-Excedance-Letters}\\ 
Shift  $\sigma_{k_{2}},\sigma_{k_{3}},\ldots,\sigma_{k_{7}}$
one non-excedance to the right; 
see Fig. \ref{Fig_exm_2}(b) for an illustration.
This results in the sequence: $3~10~1~14~7*8~9~2~5~13~4~6~12~11$, 
where $\ast$ is the $k_{2}$-th letter.  
\\
\textbf{Step ii. Placing $n$}\\
Set the $k_{2}$-th letter to be $n=15$;  
see Fig. \ref{Fig_exm_2}(c) for an illustration.
Then we obtain $\phi_{15}^{\text{den}}(\sigma,9)=3~10~1~14~7~15~8~9~2~5~13~4~6~12~11$.
}\end{example}
\vspace{-6pt}
To establish Theorem \ref{Thm-New-bijection}, 
we prove two lemmas.
\begin{lemma}\label{lemma:Exc-den}
For $\sigma\in\mathfrak{S}_{n-1}$ and $0\leq c\leq n-1$, 
we have 
\begin{align*}
	\emph{\textsf{Exc}}(\phi_{n}^{\emph{\text{den}}}(\sigma,c))&=
	\left\{
	\begin{aligned}
		&\emph{\textsf{Exc}}(\sigma), \quad\quad\quad\quad~\text{if~}0\leq c\leq \emph{\textsf{exc}}(\sigma),\\
		&\emph{\textsf{Exc}}(\sigma)\cup\{k_{d}\},\quad\text{~otherwise}
	\end{aligned}
	\right. \\
	\emph{\textsf{den}}(\phi_{n}^{\emph{\text{den}}}(\sigma,c))&=\emph{\textsf{den}}(\sigma)+c, 
\end{align*}
where $d=c-\emph{\textsf{exc}}(\sigma)$ and
$k_{d}$ is the $d$-th non-excedance of $\sigma$ from smallest to largest.
\end{lemma} 
\begin{proof} 
We keep the notation of the construction of $\phi_{n}^{\text{den}}$.
Assume that $w=\phi_{n}^{\text{den}}(\sigma,c)$. 
\\
\noindent \textbf{Case 1:} $c=0$. 
Since $w=\sigma_{1}\sigma_{2}\ldots\sigma_{n-1}n$,
it is clear that 
\begin{align*}
\textsf{Exc}(w)&=\textsf{Exc}(\sigma), \\
\textsf{den}(w)&=\textsf{den}(\sigma)=\textsf{den}(\sigma)+c,
\end{align*}
and the result holds in this case.

\noindent\textbf{Case 2:} $1\leq c\leq s$. 	
Recall that $d=s+1-c$, 
and that $e_{j_{1}}<e_{j_{2}}<\cdots<e_{j_{x}}$ are the excedance-letters of $\sigma$ that lie strictly to the left of $\sigma_{e_{d}}$  and are at least $e_{d}$.
Let $\sigma^{\prime}$ be the resulting sequence after Step 1;
that is, $\sigma^{\prime}$ is obtained from $\sigma$ by replacing $e_{j_{i}}$ with $e_{j_{i+1}}$ for $i=1,2,\ldots,x$, where $e_{j_{x+1}}=n$.
We claim that
\vspace{-6pt}
\begin{equation}\label{eq-Excl-increase}
\textsf{inv}(\textsf{\scriptsize{EXCL}}(\sigma^{\prime}))=\textsf{inv}(\textsf{\scriptsize{EXCL}}(\sigma))+ c-x.
\vspace{-6pt}
\end{equation}
\begin{itemize}[leftmargin=1.6em]
	\item[]\emph{Proof}. 
	Consider the permutation $\sigma$.
	Let $A$ be the set of positions of letters $e_{1},e_{2},\ldots,e_{d-1}$.
	Let $B$ (resp. $C$) be the set of positions of the letters in $\{e_{d},e_{d+1},\ldots,e_{s}\}$ that lie strictly to the left (resp. weakly to the right) of $\sigma_{e_{d}}$ in $\sigma$.
	Note that $B$ is the set of positions of letters $e_{j_{1}},e_{j_{2}},\ldots,e_{j_{x}}$.
	So $|B|=x$ and $|C|=s-d+1-|B|=c-x$.
	Let $\textsf{inv}_{\sigma}(X,Y)$ denote the number of inversions $(i,j)$ in $\sigma$ so that $i\in X, j\in Y$ or $i\in Y,j\in X$.
	Clearly,
	$\textsf{inv}(\textsf{\scriptsize{EXCL}}(\sigma))=\textsf{inv}_{\sigma}(A,A)+\textsf{inv}_{\sigma}(A,B)+\textsf{inv}_{\sigma}(A,C)+\textsf{inv}_{\sigma}(B,B)+\textsf{inv}_{\sigma}(B,C)+\textsf{inv}_{\sigma}(C,C).$
	It is easy to see that Step 1 does not affect $\textsf{inv}_{\sigma}(A,A),\textsf{inv}_{\sigma}(A,C),\text{~and~}\textsf{inv}_{\sigma}(C,C)$.
	Note that $\text{Step 1}$ preserves the relative order of the letters at the positions in $B$,
	so  $\textsf{inv}_{\sigma^{\prime}}(B,B)=\textsf{inv}_{\sigma}(B,B)$.
	It is not hard to see that
	$\textsf{inv}_{\sigma^{\prime}}(A,B)=\textsf{inv}_{\sigma}(A,B)$,
	since‌ $\sigma_{i}<\sigma_{j}$ if $i\in A$ and $j\in B$.
	It remains to examine the effect of $\text{Step 1}$ on $\textsf{inv}_{\sigma}(B,C)$.
	Note that the following facts hold:
	(1) $i<j$ if $i\in B$ and $j\in C$,
	(2) $e_{j_{1}}=e_{d}<\sigma_{j}$ for all $j\in C$, and
	(3) $e_{j_{x+1}}=n>\sigma_{j}$ for all $j\in C$.  
	Then  $\textsf{inv}_{\sigma^{\prime}}(B,C)=\textsf{inv}_{\sigma}(B,C)+|C|=\textsf{inv}_{\sigma}(B,C)+c-x$.
	It follows that $\textsf{inv}(\textsf{\scriptsize{EXCL}}(\sigma^{\prime}))=\textsf{inv}(\textsf{\scriptsize{EXCL}}(\sigma))+ c-x$,
	as claimed.
\end{itemize}

\noindent Clearly, Step 1 does not affect $\textsf{Exc}(\sigma)$ and $\textsf{inv}(\textsf{\scriptsize{NEXCL}}(\sigma))$.
Before considering the effect of Steps 2 and 3, we make the following two observations.

\noindent 1. There are $x$ letters in $\{\sigma_{k_{y}},\sigma_{k_{y+1}},\ldots,\sigma_{k_{t}}\}$ that are less than $e_{d}$. 

\begin{itemize}[leftmargin=1.6em]
	\item[]\emph{Proof}. 
	We need the following result of Dumont \cite{Dumont-1974}:
	Let $\sigma$ be a permutation and $v$ a positive integer, then
	$|\{\sigma_{i}:\sigma_{i}<v\leq i\}|
	=|\{\sigma_{i}:\sigma_{i}\geq v>i\}|.$
	Setting $v=e_{d}$ and noting that 
	$\{\sigma_{i}:\sigma_{i}\geq e_{d}>i\}=\{e_{j_{1}},e_{j_{2}},\ldots,e_{j_{x}}\}$,
	we have $|\{\sigma_{i}:\sigma_{i}<e_{d}\leq i\}|=x$,
	which means that there are $x$ non-excedance-letters that lie weakly to the right of $\sigma_{e_{d}}$ and are less than $e_{d}$.
	This yields the desired result.
\end{itemize}

\noindent 2. Any non-excedance-letter of $\sigma$ that lies to the left of $\sigma_{e_{d}}$ is less than $e_{d}$.
\begin{itemize}[leftmargin=1.6em]
	\item[]\emph{Proof}. 
	Let $\sigma_{i}$ be a non-excedance-letter of $\sigma$ that lies to the left of $\sigma_{e_{d}}$. Then we have $\sigma_{i}\leq i< e_{d}$, as desired.
\end{itemize}	
\noindent 
It follows from the two observations above that Steps 2 and 3 increase $\textsf{inv}(\textsf{\scriptsize{NEXCL}}(\sigma^{\prime}))$ by $x$.
Moreover, Steps 2 and 3 do not affect $\textsf{Exc}(\sigma^{\prime})$ and 
$\textsf{inv}(\textsf{\scriptsize{EXCL}}(\sigma^{\prime}))$.
In summary,
Step 1 only increases $\textsf{inv}(\textsf{\scriptsize{EXCL}}(\sigma))$ by $c-x$, while Steps 2 and 3 only increase $\textsf{inv}(\textsf{\scriptsize{NEXCL}}(\sigma^{\prime}))$ by $x$.
Then 
\begin{align*}
\textsf{Exc}(w)&=\textsf{Exc}(\sigma), \\
\textsf{den}(w)&=\textsf{den}(\sigma)+(c-x)+x=\textsf{den}(\sigma)+c.
\end{align*}
Therefore, the result holds in this case.

\noindent \textbf{Case 3:} $s+1\leq c\leq n-1$. 
Recall that $d=c-s$.
Assume that there are $u$ (resp. $v$) excedance-letters of $\sigma$ to the left (resp. right) of  $\sigma_{k_{d}}$.
Thus, $u+v=s$.
It is not hard to see that $u+d=k_{d}$.
After Steps i and ii, we 

\noindent 1. create a new excedance $k_{d}$,

\noindent 2. increase $\textsf{inv}(\textsf{\scriptsize{EXCL}}(\sigma))$ by $v$, and

\noindent 3. do not change $\textsf{inv}(\textsf{\scriptsize{NEXCL}}(\sigma))$.

\noindent Hence, we have
\vspace{-8pt}
\begin{align*}
\textsf{Exc}(w)&=\textsf{Exc}(\sigma)\cup\{k_{d}\}, \\
\textsf{den}(w)&=\textsf{den}(\sigma)+k_{d}+v=\textsf{den}(\sigma)+u+d+v\\
&=\textsf{den}(\sigma)+s+d=\textsf{den}(\sigma)+c.
\end{align*}
Thus, the result holds in this case.
\end{proof}
\begin{lemma}\label{lemma:bijection}
The map 
$$\phi_{n}^{\emph{\text{den}}}: \mathfrak{S}_{n-1}\times\{0,1,\ldots,n-1\}\rightarrow\mathfrak{S}_{n}
$$
is a bijection.
\end{lemma}	
\begin{proof}
For any positive integers $a$ and $b$ with $a<b$, 
let $(a,b)$ denote the set of all integers $i$ with $a<i<b$, and let $[a, b):=(a, b)\cup\{a\}$.
Let $w=w_{1}w_{2}\ldots w_{n}\in\mathfrak{S}_{n}$,
and let $w_{i}$ be a non-excedance-letter of $w$ 
(so $w_{i}\leq i$).
We call $w_{i}$ called a 
\emph{critical non-excedance-letter} if $[w_{i},i)\subseteq \textsf{Exc}(w)$.
Note that if $w_{i}=i$, then $w_{i}$ is a critical non-excedance-letter,
since $[w_{i},i)=\emptyset\subseteq \textsf{Exc}(w)$.
Any permutation $w\in\mathfrak{S}_{n}$ has at least one non-excedance-letter---specifically, $w_{n}$---and therefore has at least one critical non-excedance-letter, 
as the leftmost non-excedance-letter must be critical.

We keep the notation of the construction of $\phi_{n}^{\text{den}}$.
Let $w=w_{1}w_{2}\ldots w_{n}=\phi_{n}^{\text{den}}(\sigma,c)$. 
Assume that $w_{z}=n$, and that the rightmost critical non-excedance-letter of $w$ is letter $a$.
The image $w$ satisfies the following properties.

\begin{itemize}
\item[$\bullet$] In Case 1, we have $z=n$.

\item[$\bullet$] In Case 2, we have $z<a=e_{d}$.

\emph{Proof.} 
Recall that the letter $n$ in $w$ arises from replacing  $e_{j_{x}}$ in $\sigma$, 
with $e_{j_{x}}$ lying to the left of $\sigma_{e_{d}}$ in $\sigma$.
Hence $z<e_{d}$. 
In the remainder of the proof, we show that $a=e_{d}$.
Note that the non-excedance-letters of $w$ weakly to the right  of  $w_{e_{d}}$ are $w_{k_{y}},w_{k_{y+1}},\ldots,w_{k_{t+1}}$,
where $w_{k_{y}}=e_{d}$ and $w_{k_{i}}=\sigma_{k_{i-1}}$ for $y+1\leq i\leq t+1$.
We prove that $a=e_{d}$ by showing that the non-excedance-letter $w_{k_{y}}=e_{d}$ is critical, while $w_{k_{i}}$ is not critical for $i=y+1,y+2,\ldots,t+1$.
Since $\sigma_{k_{y}}$ is the first (i.e., leftmost) non-excedance-letter that lies weakly to the right of $\sigma_{e_{d}}$,
it follows that $[w_{k_{y}},k_{y})=[e_{d},k_{y})\subseteq\textsf{Exc}(\sigma)=\textsf{Exc}(w)$.
So $w_{k_{y}}=e_{d}$ is a critical non-excedance-letter of $w$.
Consider $w_{k_{i}}$ with $y+1\leq i\leq t+1$.
Note that $w_{k_{i}}=\sigma_{k_{i-1}}$.
Since $\sigma_{k_{i-1}}$ is a non-excedance-letter of $\sigma$,
we have $\sigma_{k_{i-1}}\leq k_{i-1}$.
Then $w_{k_{i}}=\sigma_{k_{i-1}}\leq k_{i-1}<k_{i}$.
Since $k_{i-1}\notin\textsf{Exc}(\sigma)=\textsf{Exc}(w)$ and $k_{i-1}\in[w_{k_{i}},k_{i})$,
we see that $w_{k_{i}}$ is not a critical non-excedance-letter of $w$,
completing the proof.

\item[$\bullet$] In Case 3, we have $a\leq z<n$.

\emph{Proof.} Observe that in this case $z=k_{d}$.  
Clearly, $k_{d}<n$.
We show below that $a\leq k_{d}$.
Note that the non-excedance-letters of $w$ are
$w_{k_{1}},w_{k_{2}},\ldots,w_{k_{d-1}},w_{k_{d+1}},\ldots,w_{k_{t+1}}$,
where $w_{k_{i}}=\sigma_{k_{i}}$ for $1\leq i\leq d-1$, and $w_{k_{i}}=\sigma_{k_{i-1}}$ for $d+1\leq i\leq t+1$.
Assume that $a=w_{k_{j}}$. 
Our goal is to show that $w_{k_{j}}\leq k_{d}$.
If $1\leq j\leq d-1$, then
$w_{k_{j}}\leq k_{j}<k_{d}$, and the result holds.
If $j=d+1$, $w_{k_{j}}=w_{k_{d+1}}=\sigma_{k_{d}}\leq k_{d}$, 
and the result holds.
If $d+2\leq j\leq t+1$, then $w_{k_{j}}=\sigma_{k_{j-1}}$.
Since $\sigma_{k_{j-1}}\leq k_{j-1}< k_{j}$,
we have $k_{j-1}\in[\sigma_{k_{j-1}},k_{j})=[w_{k_{j}},k_{j})$.
Combining this with that $k_{j-1}\notin\textsf{Exc}(w)$ (note that $d+1\leq j-1\leq t$),
we see that $w_{k_{j}}$ is not a critical non-excedance-letter of $w$,
which is a contradiction.
\end{itemize}
Based on the above properties of $w$, 
we can uniquely recover $(\sigma,c)$ from its image $w$ as follows.
\begin{itemize}
\item[$\bullet$] If $z=n$,
let $\sigma$ be obtained from $w$ by removing the letter $n$.

\item[$\bullet$] If $z<a$,  denote $e_{d}:=a$.
Assume that $e_{j_{2}}<\cdots<e_{j_{x}}<e_{j_{x+1}}:=n$ are the excedance-letters of $w$ that lie strictly to the left of $w_{e_{d}}$ and that are greater than $e_{d}$.
Let $\sigma$ be obtained from $w$ by the following three steps:

1. Delete the letter $e_{d}$ and replace it with a star ``$\ast$'';

2. Shift the non-excedance-letters that lie to the right of the star one non-excedance to the left (we think of the position occupied by the star as a non-excedance);

3. Replace $e_{j_{i}}$ with $e_{j_{i-1}}$ for all $i=2,3,\ldots,x+1$, 
where we set $e_{j_{1}}=e_{d}$.

\item[$\bullet$] If $a\leq z<n$,
let $\sigma$ be obtained from $w$ by the following two steps:

1. Delete the letter $n$ and replace it with a star ``$\ast$'';

2. Shift the non-excedance-letters that lie to the right of the star one non-excedance to the left (we think of the position occupied by the star as a non-excedance).
\end{itemize}
Now we have recovered $\sigma$ from $w$.
Finally, let $c=\textsf{den}(w)-\textsf{den}(\sigma)$,
and we successfully recover $(\sigma,c)$ from $w$.
Therefore, the map $\phi_{n}^{\text{den}}: \mathfrak{S}_{n-1}\times\{0,1,\ldots,n-1\}\rightarrow\mathfrak{S}_{n}
$ is an injection.
Since $|\mathfrak{S}_{n-1}\times\{0,1,\ldots,n-1\}|=|\mathfrak{S}_{n}|=n!$,
it follows that $\phi_{n}^{\text{den}}$ is a bijection.
\end{proof}
By Lemmas \ref{lemma:Exc-den} and \ref{lemma:bijection},
we obtain Theorem \ref{Thm-New-bijection}.

\section{Proof of Theorem \ref{mian-theorem-(excr,den)}}\label{Proof of Theorem 1.3}
As we will see in this section, 
the bijection $\phi_{n}^{\text{den}}$ introduced in the previous section can be used not only to prove that $(\textsf{exc},\textsf{den})$ is Euler--Mahonian, 
but also to show that  $(\textsf{exc}_{r},\textsf{den})$ is $r$-Euler--Mahonian.

To prove that a pair $(\textsf{stat}_{1},\textsf{stat}_{2})$ of permutation statistics is $r$-Euler--Mahonian, 
it suffices to show that if $n\leq r$,   $(\textsf{stat}_{1},\textsf{stat}_{2})$ is equidistributed with $(\textsf{0},\textsf{inv})$ over $\mathfrak{S}_{n}$;
if $n>r$, there exists a bijection 
\begin{equation*} \phi_{r,n}:\mathfrak{S}_{n-1}\times\{0,1,\ldots,n-1\}\rightarrow\mathfrak{S}_{n}
\end{equation*} 
such that for $\sigma\in\mathfrak{S}_{n-1}$ and $0\leq c\leq n-1$, 
we have
\begin{equation}\label{eq-r-Euler--Mahonian}
	\begin{split}
		\textsf{stat}_{1}(\phi_{r,n}(\sigma,c))&=\left
		\{
		\begin{aligned}
			&\textsf{stat}_{1}(\sigma), \quad\quad\quad\text{if~}0\leq c\leq \textsf{stat}_{1}(\sigma)+r-1,\\
			&\textsf{stat}_{1}(\sigma)+1,\quad\text{~otherwise,}
		\end{aligned}
		\right.\\ 
		\textsf{stat}_{2}(\phi_{r,n}(\sigma,c))&=\textsf{stat}_{2}(\sigma)+c,
	\end{split}
\end{equation} 
see \cite{Rawlings-1981}.

Clearly,  when $n\leq r$, 
we have $\textsf{exc}_{r}(\sigma)=0$ for any $\sigma\in\mathfrak{S}_{n}$. 
Note that $\textsf{den}$ and $\textsf{inv}$ are equidistributed.
Thus, 
$(\textsf{exc}_{r},\textsf{den})$ and $(\textsf{0},\textsf{inv})$ are equidistributed over $\mathfrak{S}_{n}$ when $n\leq r$.
In the remainder of this section,  we assume that $n>r$.

Let $\sigma\in\mathfrak{S}_{n-1}$ and $0\leq c\leq n-1$.
Recall that
$$\textsf{Exc}_{r}(\sigma)=\{i:\sigma_{i}>i,\!~i\geq r\}=\{i\in\textsf{Exc}(\sigma):i\geq r\}.$$
Denote 
\vspace{-8pt}
\begin{align*}
A_{r}(\sigma)=\{i:\sigma_{i}>i,\!~i\leq r-1\}
\text{~and~}
B_{r}(\sigma)=\{i:\sigma_{i}\leq i,\!~i\leq r-1\}.
\end{align*}
Let $w=\phi_{n}^{\text{den}}(\sigma,c)$.
By Theorem \ref{Thm-New-bijection}, we have
\begin{align*}
	\textsf{Exc}(w)&=
	\left\{
	\begin{aligned}
		&\textsf{Exc}(\sigma), \quad\quad\quad\quad~\text{if~}0\leq c\leq \textsf{exc}(\sigma),\\
		&\textsf{Exc}(\sigma)\cup\{k_{d}\},\quad\text{~otherwise}
	\end{aligned}
	\right. \\
	\textsf{den}(w)&=\textsf{den}(\sigma)+c, 
\end{align*}
where $d=c-\textsf{exc}(\sigma)$ and $k_{d}$ is the $d$-th non-excedance of $\sigma$ from smallest to largest.
Then
\begin{itemize}
	\item[$\bullet$]If $0\leq c\leq \textsf{exc}(\sigma)$, we have
	$\textsf{Exc}(w)=\textsf{Exc}(\sigma)$,
	implying that $\textsf{Exc}_{r}(w)=\textsf{Exc}_{r}(\sigma)$.
	\item[$\bullet$]If $\textsf{exc}(\sigma)+1\leq c\leq \textsf{exc}(\sigma)+|B_{r}(\sigma)|$,
	we have $\textsf{Exc}(w)=\textsf{Exc}(\sigma)\cup\{k_{d}\}$.
	As $d=c-\textsf{exc}(\sigma)$, 
	then $1\leq d\leq |B_{r}(\sigma)|$. 
	It follows that $k_{d}\in B_{r}(\sigma)$, and hence $k_{d}\leq r-1$.
	So $k_{d}\notin\textsf{Exc}_{r}(w)$.
	Combining this with $\textsf{Exc}(w)=\textsf{Exc}(\sigma)\cup\{k_{d}\}$,
	we have $\textsf{Exc}_{r}(w)=\textsf{Exc}_{r}(\sigma)$.
	
	\item[$\bullet$]If $c>\textsf{exc}(\sigma)+|B_{r}(\sigma)|$,
	we have $\textsf{Exc}(w)=\textsf{Exc}(\sigma)\cup\{k_{d}\}$.
	Note that $d=c-\textsf{exc}(\sigma)>|B_{r}(\sigma)|$. 
	It follows that $k_{d}\notin B_{r}(\sigma)$,
	and hence $k_{d}\geq r$. 
    So $k_{d}\in \textsf{Exc}_{r}(w)$.
	Combining this with $\textsf{Exc}(w)=\textsf{Exc}(\sigma)\cup\{k_{d}\}$,
	we have $\textsf{Exc}_{r}(w)=\textsf{Exc}_{r}(\sigma)\cup\{k_{d}\}$.
\end{itemize}	
Thus, in summary, 
\begin{equation} \label{eq-Excr-den-Br}
	\begin{split}
		\textsf{Exc}_{r}(w)&=\left
		\{
		\begin{aligned}
			&\textsf{Exc}_{r}(\sigma), \quad\quad\quad\quad~\text{if~}0\leq c\leq \textsf{exc}(\sigma)+|B_{r}(\sigma)|,\\
			&\textsf{Exc}_{r}(\sigma)\cup\{k_{d}\},\quad\text{~otherwise,}
		\end{aligned}
		\right.\\
				\textsf{den}(w)&=\textsf{den}(\sigma)+c.
	\end{split}
\end{equation}
It is clear that $\textsf{Exc}(\sigma)$ is the disjoint union of $\textsf{Exc}_{r}(\sigma)$ and $A_{r}(\sigma)$. Thus,
$$\textsf{exc}(\sigma)=\textsf{exc}_{r}(\sigma)+|A_{r}(\sigma)|.$$
It follows that
\vspace{-6pt}
\begin{align*}
\textsf{exc}(\sigma)+|B_{r}(\sigma)|&=\textsf{exc}_{r}(\sigma)+|A_{r}(\sigma)|+|B_{r}(\sigma)|=\textsf{exc}_{r}(\sigma)+r-1.
\end{align*}
Then (\ref{eq-Excr-den-Br}) becomes
\begin{equation*} 
	\begin{split}
		\textsf{Exc}_{r}(w)&=\left
		\{
		\begin{aligned}
			&\textsf{Exc}_{r}(\sigma), \quad\quad\quad\quad~\text{if~}0\leq c\leq \textsf{exc}_{r}(\sigma)+r-1,\\
			&\textsf{Exc}_{r}(\sigma)\cup\{k_{d}\},\quad\text{~otherwise,}
		\end{aligned}
		\right.\\ 
		\textsf{den}(w)&=\textsf{den}(\sigma)+c,
	\end{split}
\end{equation*}
implying that
\vspace{-6pt}
\begin{equation}\label{eq-(excr,den)-all}
	\begin{split}
		\textsf{exc}_{r}(w)&=\left
		\{
		\begin{aligned}
			&\textsf{exc}_{r}(\sigma), \quad\quad\quad\text{if~}0\leq c\leq \textsf{exc}_{r}(\sigma)+r-1,\\
			&\textsf{exc}_{r}(\sigma)+1,\quad\text{~otherwise,}
		\end{aligned}
		\right.\\ 
		\textsf{den}(w)&=\textsf{den}(\sigma)+c.
	\end{split}
\end{equation}
Comparing with equation (\ref{eq-r-Euler--Mahonian}),
we see that $(\textsf{exc}_{r},\textsf{den})$ is $r$-Euler--Mahonian,
completing the proof of Theorem \ref{mian-theorem-(excr,den)}.

\begin{example}
\emph{
Let $\sigma=621534$.
We have $\textsf{Exc}(\sigma)=\{1,4\}$ and $\textsf{den}(\sigma)=7$. 
Clearly,
$$\textsf{exc}(\sigma)=2,\!~
\textsf{exc}_{2}(\sigma)=1,\!~
\textsf{exc}_{3}(\sigma)=1,\!~
\textsf{exc}_{4}(\sigma)=1,\!~
\textsf{exc}_{5}(\sigma)=0,\!~
\textsf{exc}_{6}(\sigma)=0.$$
In Table \ref{Table-1}, 
we list the images of $(\sigma,c)$ under the bijection $\phi_{7}^{\text{den}}$ for $c=0,1,\ldots,6$,
together with the values of $(\textsf{exc}_{r},\textsf{den})$ of $\phi_{7}^{\text{den}}(\sigma,c)$ for $r=1,2,\ldots,6$.
From Table \ref{Table-1}, one can verify equation (\ref{eq-(excr,den)-all}).
}\end{example}
\begin{table}[t]
	\centering
	\caption{Images of $(\sigma,c)$ under $\phi_{7}^{\text{den}}$ and values of $(\textsf{exc}_{r},\textsf{den})$}\label{Table-1}	
	\begin{tabular}{c|c|c|c|c|c|c|c}
			$c$
			&$\phi_{7}^{\text{den}}(\sigma,c)$
			&$(\textsf{exc},\textsf{den})$
			&$(\textsf{exc}_{2},\textsf{den})$
			&$(\textsf{exc}_{3},\textsf{den})$
			&$(\textsf{exc}_{4},\textsf{den})$
			&$(\textsf{exc}_{5},\textsf{den})$
			&$(\textsf{exc}_{6},\textsf{den})$
			\\
			\hline
			$0$&$6215347$&$(2,7)~\!~$&$(1,7)~\!~$&$(1,7)~\!~$&$(1,7)~\!~$&$(0,7)~\!~$&$(0,7)~\!~$ \\
			$1$&$7215364$&$(2,8)~\!~$&$(1,8)~\!~$&$(1,8)~\!~$&$(1,8)~\!~$&$(0,8)~\!~$&$(0,8)~\!~$  \\
			$2$&$7216534$&$(2,9)~\!~$&$(1,9)~\!~$&$(1,9)~\!~$&$(1,9)~\!~$&$(0,9)~\!~$&$(0,9)~\!~$ \\
			$3$&$6725134$&$(3,10)$&$(2,10)$&$(1,10)$&$(1,10)$&$(0,10)$&$(0,10)$ \\ 
			$4$&$6275134$&$(3,11)$&$(2,11)$&$(2,11)$&$(1,11)$&$(0,11)$&$(0,11)$    \\
			$5$&$6215734$&$(3,12)$&$(2,12)$&$(2,12)$&$(2,12)$&$(1,12)$&$(0,12)$    \\
			$6$&$6215374$&$(3,13)$&$(2,13)$&$(2,13)$&$(2,13)$&$(1,13)$&$(1,13)$    \\
		\end{tabular}
\end{table} 
\section{Proof of Theorem \ref{mian-general-theorem}}\label{Proof of Theorem 1.4}
In this section, we present an extension of the bijection 
$\phi_{n}^{\text{den}}$ introduced in Section \ref{section-new-proof},
followed by a proof of Theorem \ref{mian-general-theorem}.

Recall that for $\sigma\in\mathfrak{S}_{n}$, if $\sigma_{i}\geq i+g$ and $\sigma_{i}\geq h$,
the letter $\sigma_{i}$ is called a $g$-gap $h$-level excedance-letter of $\sigma$,
and the position $i$ is called a $g$-gap $h$-level excedance-letter position.
Recall also  that
\begin{align*}
	g\textsf{den}_{h}(\sigma)=\sum_{i\in g\textsf{Exclp}_{h}(\sigma)}(i+g-1)+\textsf{inv}(g\textsf{\scriptsize{EXCL}}_{h}(\sigma))+\textsf{inv}(g\textsf{\scriptsize{NEXCL}}_{h}(\sigma)),
\end{align*}
where 
$g\textsf{\scriptsize{EXCL}}_{h}(\sigma)$ (resp. $g\textsf{\scriptsize{NEXCL}}_{h}(\sigma)$) is the subsequence of $\sigma$  that consists of the $g$-gap $h$-level excedance-letters (resp.  remaining letters),
and $g\textsf{Exclp}_{h}(\sigma)$ is the set of $g$-gap $h$-level excedance-letter positions of $\sigma$ (i.e., the positions corresponding to the letters of $g\textsf{\scriptsize{EXCL}}_{h}(\sigma)$).
The main result of this section is the following theorem.
\begin{theorem}\label{Thm-bijection-denl}
	Let $g\geq1$ and $1\leq h\leq n$.
	There is a bijection $$\phi_{g,h,n}^{\emph{\text{den}}}: \mathfrak{S}_{n-1}\times\{0,1,\ldots,n-1\}\rightarrow\mathfrak{S}_{n},$$
	such that for $\sigma\in\mathfrak{S}_{n-1}$ and $0\leq c\leq n-1$, 
	we have 
\begin{equation}\label{eq-(Excr,denl)-all}
\begin{split}
g\emph{\textsf{Exclp}}_{h}(\phi_{g,h,n}^{\emph{\text{den}}}(\sigma,c))&=
\left\{
\begin{aligned}
	&g\emph{\textsf{Exclp}}_{h}(\sigma), \quad\quad\quad\quad~\text{if~}0\leq c\leq |g\emph{\textsf{Exclp}}_{h}(\sigma)|+g-1,\\
	&g\emph{\textsf{Exclp}}_{h}(\sigma)\cup\{k_{d}\},\quad\text{~otherwise}
\end{aligned}
\right.\\
g\emph{\textsf{den}}_{h}(\phi_{g,h,n}^{\emph{\text{den}}}(\sigma,c))&=g\emph{\textsf{den}}_{h}(\sigma)+c,
	\end{split}
\end{equation}	
where $d=c-|g\emph{\textsf{Exclp}}_{h}(\sigma)|-g+1$ and
$k_{d}$ is the $d$-th non-$g$-gap $h$-level excedance-letter position of $\sigma$ from smallest to largest.
\end{theorem}

\vskip 0pt
\begin{framed}
\begin{center}
\vskip -8pt
{\bf{The map}} $\phi_{g,h,n}^{\text{den}}: \mathfrak{S}_{n-1}\times\{0,1,\ldots,n-1\}\rightarrow\mathfrak{S}_{n}$ for $1\leq h\leq n$
\end{center}
\vskip -8pt

	Let $\sigma=\sigma_{1}\sigma_{2}\ldots\sigma_{n-1}\in\mathfrak{S}_{n-1}$ and let $c$ be an integer with $0\leq c\leq n-1$.
	Let $|g\textsf{Exclp}_{h}(\sigma)|=s$, and let the $g$-gap $h$-level excedance-letters of $\sigma$ be $e_{1},e_{2},\ldots,e_{s}$,
	where $e_{1}<e_{2}<\cdots<e_{s}$.
	Let the non-$g$-gap $h$-level excedance-letters of $\sigma$ be $\sigma_{k_{1}},\sigma_{k_{2}},\ldots,\sigma_{k_{t}}$ (equivalently, the non-$g$-gap $h$-level excedance-letter positions of $\sigma$ are $k_{1},k_{2},\ldots,k_{t}$),
	where $k_{1}<k_{2}<\cdots<k_{t}$.
	We distinguish three cases.
	\begin{itemize}[leftmargin=1.5em]
		\item[$\bullet$]
		Case 1: $0\leq c\leq g-1$.  
		Define $\phi_{g,h,n}^{\text{den}}(\sigma,c)$ to be the permutation obtained from $\sigma$ by inserting $n$ immediately after the letter $\sigma_{n-1-c}$.
		
		\item[$\bullet$] 
		Case 2: $g\leq c\leq s+g-1$.  
		Let $d=s+g-c$ and $p=e_{d}-g+1$.
		Assume that $e_{j_{1}}<e_{j_{2}}<\cdots<e_{j_{x}}$ are the $g$-gap $h$-level excedance-letters of $\sigma$ that lie strictly to the left of $\sigma_{p}$ and are at least $e_{d}$.
		Clearly, $e_{j_{1}}=e_{d}$.
		Assume that $\sigma_{k_{y}},\sigma_{k_{y+1}},\ldots,\sigma_{k_{t}}$ are the non-$g$-gap $h$-level excedance-letters of $\sigma$ that lie weakly to the right of $\sigma_{p}$.
		Define $\phi_{g,h,n}^{\text{den}}(\sigma,c)$  as the permutation obtained from $\sigma$ by the following three steps:
		
		\textbf{\text{Step 1}}: Replace $e_{j_{i}}$ with $e_{j_{i+1}}$ for  $i=1,2,\ldots,x$, with the convention that $e_{j_{x+1}}=n$.
		
		\textbf{\text{Step 2}}: Replace $\sigma_{k_{i}}$ with $\sigma_{k_{i-1}}$ for $i=y+1,y+2,\ldots,t,t+1$, with the convention that $k_{t+1}=n$
		(in other words, shift the letters $\sigma_{k_{y}},\sigma_{k_{y+1}},\ldots,\sigma_{k_{t}}$ one non-$g$-gap $h$-level excedance-letter position to the right).
		
		\textbf{\text{Step 3}}: Set the $k_{y}$-th letter to be $e_{d}$.  
		
		\item[$\bullet$] Case 3: $s+g\leq c\leq n-1$. 
		Let $d=c-s-g+1$. 
		Define $\phi_{g,h,n}^{\text{den}}(\sigma,c)$  to be the permutation obtained from $\sigma$ by the following two steps:
		
		\textbf{\text{Step i}}: Replace $\sigma_{k_{i}}$ with $\sigma_{k_{i-1}}$ for $i=d+1,d+2,\ldots,t,t+1$, with the convention that $k_{t+1}=n$
		(in other words, shift the letters $\sigma_{k_{d}},\sigma_{k_{d+1}},\ldots,\sigma_{k_{t}}$  one non-$g$-gap $h$-level excedance-letter position to the right).
		
		\textbf{\text{Step ii}}: Set the $k_{d}$-th letter to be $n$.    
	\end{itemize}
\end{framed}
\begin{proof}[Proof of Theorem \ref{Thm-bijection-denl}]
We keep the notation of the construction of $\phi_{g,h,n}^{\text{den}}$.
Let $w=\phi_{g,h,n}^{\text{den}}(\sigma,c)$. 
We first prove equation (\ref{eq-(Excr,denl)-all}).
\\
\noindent \textbf{Case 1:} $0\leq c\leq g-1$. 
Note that $w=\sigma_{1}\sigma_{2}\ldots\sigma_{n-1-c}n\sigma_{n-c}\ldots\sigma_{n-1}$.
It is not hard to see that the letters 
$w_{n-c},w_{n-c+1},\ldots,w_{n}$
(i.e., the letters $n,\sigma_{n-c},\ldots,\sigma_{n-1}$)
are all non-$g$-gap $h$-level excedance-letters of $w$.
Indeed, for $n-c\leq i\leq n$,
we have $w_{i}\leq n\leq i+c<i+g$.
Similarly, $\sigma_{n-c},\sigma_{n-c+1},\ldots,\sigma_{n-1}$ are all non-$g$-gap $h$-level excedance-letters of $\sigma$.
Therefore, $\textsf{inv}(g\textsf{\scriptsize{NEXCL}}_{h}(w))-\textsf{inv}(g\textsf{\scriptsize{NEXCL}}_{h}(\sigma))=c$.
It is not hard to see that
\begin{align*}
	g\textsf{Exclp}_{h}(w)&=g\textsf{Exclp}_{h}(\sigma), \\
	g\textsf{den}_{h}(w)&=g\textsf{den}_{h}(\sigma)+c.
\end{align*}
Thus, (\ref{eq-(Excr,denl)-all}) holds in this case.

\noindent \textbf{Case 2:} $g\leq c\leq s+g-1$. 	
Recall that $d=s+g-c$,  $p=e_{d}-g+1$, and that $e_{j_{1}}<e_{j_{2}}<\cdots<e_{j_{x}}$ are the $g$-gap $h$-level excedance-letters of $\sigma$ that lie strictly to the left of $\sigma_{p}$ and are at least $e_{d}$.
Since there are $x$ letters in $\{e_{d},e_{d+1},\ldots,e_{s}\}$ that lie strictly to the left of $\sigma_{p}$ 
(namely $e_{j_{1}},e_{j_{2}},\ldots,e_{j_{x}}$),
there are $s-d+1-x$ letters in $\{e_{d},e_{d+1},\ldots,e_{s}\}$ that lie weakly to the right of $\sigma_{p}$.
On the other hand, any $g$-gap $h$-level excedance-letters of $\sigma$ that lie weakly to the right of $\sigma_{p}$ must be greater than $e_{d}$.
Indeed, if $\sigma_{i}$ is such a letter, then $\sigma_{i}\geq i+g\geq p+g=e_{d}+1>e_{d}$.
Then we see that there are $s-d+1-x$ $g$-gap $h$-level excedance-letters of $\sigma$ that lie weakly to the right of $\sigma_{p}$, and each of them is greater than $e_{d}$.
Let $\sigma^{\prime}$ be the resulting sequence after Step 1.
That is, $\sigma^{\prime}$ is obtained from $\sigma$ by replacing $e_{j_{i}}$ with $e_{j_{i+1}}$ for $i=1,2,\ldots,x$, where $e_{j_{x+1}}=n$.
A similar argument as in the proof of (\ref{eq-Excl-increase}) yields
\begin{align*}
\textsf{inv}(g\textsf{\scriptsize{EXCL}}_{h}(\sigma^{\prime}))
-\textsf{inv}(g\textsf{\scriptsize{EXCL}}_{h}(\sigma))=s-d+1-x
=c-g-x+1.
\end{align*}
Clearly, Step 1 does not affect $g\textsf{Exclp}_{h}(\sigma)$ and $\textsf{inv}(g\textsf{\scriptsize{NEXCL}}_{h}(\sigma))$.
Before considering the effect of Steps 2 and 3, we make the following two observations.

\noindent 1. There are $x+g-1$ letters in $\{\sigma_{k_{y}},\sigma_{k_{y+1}},\ldots,\sigma_{k_{t}}\}$ that are less than $e_{d}$. 
\begin{itemize}[leftmargin=1.6em]
\item[]\emph{Proof}. 
We need the following result (see \cite{Liu-2024}):
Let $\sigma$ be a permutation, and $v,g$ be positive integers with  $v\geq g$,
then
$|\{\sigma_{i}:i\geq v-g+1,\!~\sigma_{i}<v\}|
=|\{\sigma_{i}:i<v-g+1,\!~\sigma_{i}\geq v\}|+g-1	
$.
Set $v=e_{d}$ and $p=e_{d}-g+1$. Then
$|\{\sigma_{i}:i\geq p,\!~\sigma_{i}<e_{d}\}|
=|\{\sigma_{i}:i<p,\!~\sigma_{i}\geq e_{d}\}|+g-1.$
Let $\sigma_{i}\in\{\sigma_{i}:i<p,\!~\sigma_{i}\geq e_{d}\}$,
we have $\sigma_{i}\geq e_{d}=p+g-1>i+g-1$,
then $\sigma_{i}\geq i+g$.
Combining this with $\sigma_{i}\geq e_{d}\geq h$,
we see that $\sigma_{i}$ is a $g$-gap $h$-level excedance-letter of $\sigma$.
Then we have $|\{\sigma_{i}:i<p,\!~\sigma_{i}\geq e_{d}\}|=|\{e_{j_{1}},e_{j_{2}},\ldots,e_{j_{x}}\}|=x$. 
It follows that $|\{\sigma_{i}:i\geq p,\!~\sigma_{i}<e_{d}\}|=x+g-1$.
Now we let $\sigma_{i}\in\{\sigma_{i}:i\geq p,\!~\sigma_{i}<e_{d}\}$.
Since $\sigma_{i}<e_{d}=p+g-1\leq i+g-1$,
we see that $\sigma_{i}$ is a non-$g$-gap $h$-level excedance-letter of $\sigma$.
Then the equation $|\{\sigma_{i}:i\geq p,\!~\sigma_{i}<e_{d}\}|=x+g-1$ means that there are $x+g-1$ non-$g$-gap $h$-level excedance-letters lying  weakly to the right of $\sigma_{p}$ and less than $e_{d}$. 
This yields the desired result.
\end{itemize}

\noindent 2. Any non-$g$-gap $h$-level excedance-letter of $\sigma$ that lies  to the left of $\sigma_{p}$ is less than $e_{d}$.
\begin{itemize}[leftmargin=1.6em]
	\item[]\emph{Proof}. 
	Let $\sigma_{i}$ be a non-$g$-gap $h$-level excedance-letter of $\sigma$ that lies to the left of $\sigma_{p}$.
	Then we have $i<p$ and either $\sigma_{i}\leq i+g-1$ or $\sigma_{i}<h$.
	If $\sigma_{i}\leq i+g-1$, we have $\sigma_{i}\leq i+g-1<p+g-1=e_{d}$.
	If $\sigma_{i}<h$, we have $\sigma_{i}<h\leq e_{d}$.
	Thus, $\sigma_{i}<e_{d}$ holds in each case, as desired.
\end{itemize}
It follows from the two observations above that Steps 2 and 3 increase $\textsf{inv}(g\textsf{\scriptsize{NEXCL}}_{h}(\sigma^{\prime}))$ by $x+g-1$.
Moreover, Steps 2 and 3 do not affect $g\textsf{Exclp}_{h}(\sigma^{\prime})$ and $\textsf{inv}(g\textsf{\scriptsize{EXCL}}_{h}(\sigma^{\prime}))$.
In summary,
Step 1 only increases $\textsf{inv}(g\textsf{\scriptsize{EXCL}}_{h}(\sigma))$ by $c-g-x+1$, while Steps 2 and 3 only increase $\textsf{inv}(g\textsf{\scriptsize{NEXCL}}_{h}(\sigma^{\prime}))$ by $x+g-1$.
Then we get
\begin{align*}
	g\textsf{Exclp}_{h}(w)&=g\textsf{Exclp}_{h}(\sigma), \\
	g\textsf{den}_{h}(w)&=g\textsf{den}_{h}(\sigma)+(c-g-x+1)+(x+g-1)
	=g\textsf{den}_{h}(\sigma)+c.
\end{align*}
Thus, (\ref{eq-(Excr,denl)-all}) holds in this case.

\noindent \textbf{Case 3:} $s+g\leq c\leq n-1$. 
Recall that, in this case, we have $d=c-s-g+1$.
Assume that there are $u$ (resp. $v$) $g$-gap $h$-level excedance-letters of $\sigma$ to the left (resp. right) of  $\sigma_{k_{d}}$.
Thus, $u+v=s$.
It is easy to see that $u+d=k_{d}$.
After Steps i and ii, we 

\noindent 1. create a new $g$-gap $h$-level excedance-letter position  $k_{d}$, 
\begin{itemize}[leftmargin=1.6em]
	\item[]\emph{Proof}. 
Note that $k_{d}=u+d\leq s+d=c-g+1\leq (n-1)-g+1=n-g$.
Therefore, $w_{k_{d}}=n\geq k_{d}+g$. 
Combining this with the condition that $h\leq n$,
we see that $w_{k_{d}}$ is a $g$-gap $h$-level excedance-letter of $w$,
and the proof follows.
\end{itemize}

\noindent 2. increase $\textsf{inv}(g\textsf{\scriptsize{EXCL}}_{h}(\sigma))$ by $v$, and

\noindent 3. do not change $\textsf{inv}(g\textsf{\scriptsize{NEXCL}}_{h}(\sigma))$.

\noindent Hence, we have
\begin{align*}
	g\textsf{Exclp}_{h}(w)&=g\textsf{Exclp}_{h}(\sigma)\cup\{k_{d}\},\\
	g\textsf{den}_{h}(w)&=g\textsf{den}_{h}(\sigma)+k_{d}+g-1+v=g\textsf{den}_{h}(\sigma)+u+d+g-1+v\\
	~\quad\quad\quad~&=g\textsf{den}_{h}(\sigma)+s+d+g-1=g\textsf{den}_{h}(\sigma)+c.
\end{align*}
Thus, (\ref{eq-(Excr,denl)-all}) holds in this case.

We have now completed the proof of  (\ref{eq-(Excr,denl)-all}).
It remains to show that 
$$\phi_{g,h,n}^{\text{den}}: \mathfrak{S}_{n-1}\times\{0,1,\ldots,n-1\}\rightarrow\mathfrak{S}_{n}
$$
is a bijection.
Let $w=w_{1}w_{2}\ldots w_{n}\in\mathfrak{S}_{n}$,
and let $w_{i}$ be a non-$g$-gap excedance-letter of $w$ 
(that is, $w_{i}<i+g$, and thus $w_{i}-g+1\leq i$),
we say that $w_{i}$ is a \emph{critical non-$g$-gap excedance-letter} of $w$ 
if $[w_{i}-g+1,i)\subseteq g\textsf{Exc}(w)=\{j:w_{j}\geq j+g\}$.
Any permutation $w\in\mathfrak{S}_{n}$ has at least one non-$g$-gap excedance-letter---specifically, $w_{n}$---and therefore has at least one critical non-$g$-gap excedance-letter, 
as the leftmost non-$g$-gap excedance-letter must be critical.

Let $w=w_{1}w_{2}\ldots w_{n}=\phi_{g,h,n}^{\text{den}}(\sigma,c)$.
Assume that $w_{z}=n$, and that the rightmost critical non-$g$-gap excedance-letter of $w$ is letter $a$.
The image $w$ satisfies the following properties.
\begin{itemize}
	\item[$\bullet$] In Case 1, we have $z+g> n$. 
	
	\emph{Proof}. In this case, we have $0\leq c\leq g-1$ and $z=n-c$. Then $z+g=n-c+g>n$, as desired.
	
	\item[$\bullet$] In Case 2, we have $\max\{z+g,h\}\leq a=e_{d}$.
	
\emph{Proof}. 
Recall that the letter $n$ in $w$ arises from replacing $e_{j_{x}}$ in $\sigma$, with $e_{j_{x}}$ lying to the left of $\sigma_{p}$ in $\sigma$. 
Hence  $z<p=e_{d}-g+1$,
so $z+g\leq e_{d}$.
Since $e_{d}$ is a $g$-gap $h$-level excedance-letter, we have $h\leq e_{d}$.
Therefore, $\max\{z+g,h\}\leq e_{d}$.
In the remainder of the proof, we show that $a=e_{d}$.
The non-$g$-gap $h$-level excedance-letter of $w$ weakly to the right of $w_{p}$ are $w_{k_{y}},w_{k_{y+1}},\ldots,w_{k_{t+1}}$,
where $w_{k_{y}}=e_{d}$ and $w_{k_{i}}=\sigma_{k_{i-1}}$ for $y+1\leq i\leq t+1$.
We claim that the non-$g$-gap excedance-letters of $w$ weakly to the right of $w_{p}$ are also $w_{k_{y}},w_{k_{y+1}},\ldots,w_{k_{t+1}}$.
Observe that a non-$g$-gap excedance-letter must be a non-$g$-gap $h$-level excedance-letter.
It suffices to show that 
for any $i$ with $y\leq i\leq t+1$, $w_{k_{i}}$ is a non-$g$-gap  excedance-letter of $w$.
Since $w_{k_{i}}$ is a non-$g$-gap $h$-level excedance-letter of $w$,
we have $w_{k_{i}}<k_{i}+g$ or $w_{k_{i}}<h$.
In either case, we have $w_{k_{i}}<k_{i}+g$, 
because $w_{k_{i}}<h\leq e_{d}=p+g-1\leq k_{i}+g-1<k_{i}+g$.
Thus, $w_{k_{i}}$ is a non-$g$-gap excedance-letter of $w$,
which completes the proof of the claim.
We prove that $a=e_{d}$ by showing that $w_{k_{y}}=e_{d}$ is a critical non-$g$-gap excedance-letter of $w$ and
$w_{k_{i}}$ is not a critical non-$g$-gap excedance-letter of $w$ for $i=y+1,y+2,\ldots,t+1$.
Since $\sigma_{k_{y}}$ is the first (leftmost) non-$g$-gap $h$-level excedance-letter of $\sigma$ lying weakly to the right of $\sigma_{p}$,
we have that $[w_{k_{y}}-g+1,k_{y})=[e_{d}-g+1,k_{y})=[p,k_{y})\subseteq g\textsf{Exclp}_{h}(\sigma)=g\textsf{Exclp}_{h}(w)\subseteq g\textsf{Exc}(w)$,
hence $w_{k_{y}}$ is a critical non-$g$-gap excedance-letter of $w$.
Consider $w_{k_{i}}$ with $y+1\leq i\leq t+1$.
Note that $w_{k_{i}}=\sigma_{k_{i-1}}$.
Since $\sigma_{k_{i-1}}$ is a non-$g$-gap $h$-level excedance-letter of $\sigma$,
we have $\sigma_{k_{i-1}}\leq k_{i-1}+g-1$ or $\sigma_{k_{i-1}}<h$.
In either case, we have $\sigma_{k_{i-1}}\leq k_{i-1}+g-1$, since $\sigma_{k_{i-1}}<h\leq e_{d}=p+g-1\leq k_{i-1}+g-1$.
Thus, $w_{k_{i}}-g+1=\sigma_{k_{i-1}}-g+1\leq k_{i-1}<k_{i}$.
So $k_{i-1}\in[w_{k_{i}}-g+1,k_{i})$.
By our claim we see that $k_{i-1}\notin g\textsf{Exc}(w)$ (note that $y\leq i-1\leq t$).
Thus, $w_{k_{i}}$ is not a critical non-$g$-gap excedance-letter of $w$.
We complete the proof.
	\item[$\bullet$] In Case 3, we have $a<\max\{z+g,h\}\leq n$.
    
    \emph{Proof.} Note that in this case we have $z=k_{d}$.
	In the proof of Case 3 of (\ref{eq-(Excr,denl)-all}), 
	we showed that $k_{d}\leq n-g$, and hence $k_{d}+g\leq n$.
	Combining this with the condition that $h\leq n$, 
	we obtain $\max\{k_{d}+g,h\}\leq n$. 
    In the remainder of the proof, we show that $a<\max\{k_{d}+g,h\}$.
	Note that the non-$g$-gap $h$-level excedance-letters of $w$ are
	$w_{k_{1}},w_{k_{2}},\ldots,w_{k_{d-1}},w_{k_{d+1}},\ldots,w_{k_{t+1}}$,
	where $w_{k_{i}}=\sigma_{k_{i}}$ for $1\leq i\leq d-1$,
	and $w_{k_{i}}=\sigma_{k_{i-1}}$ for $d+1\leq i\leq t+1$.
	Observe that a non-$g$-gap excedance-letter of $w$ must be a non-$g$-gap $h$-level excedance-letter of $w$.
	Assume that $a=w_{k_{j}}$,
	and our goal is to show that $w_{k_{j}}<\max\{k_{d}+g,h\}$.
	Since $w_{k_{j}}$ is non-$g$-gap excedance-letter of $w$,
	we have $w_{k_{j}}<k_{j}+g$.
	If $1\leq j\leq d-1$,  
	then $w_{k_{j}}<k_{j}+g<k_{d}+g$, and the result holds.
	If $j=d+1$, we have $w_{k_{j}}=w_{k_{d+1}}=\sigma_{k_{d}}$.
	As $\sigma_{k_{d}}$ is a non-$g$-gap $h$-level excedance-letter of $\sigma$, 
	we have $\sigma_{k_{d}}<k_{d}+g$ or $\sigma_{k_{d}}<h$.
	Therefore, $w_{k_{j}}=\sigma_{k_{d}}<\max\{k_{d}+g,h\}$,
	and the result holds.
    If $d+2\leq j\leq t+1$,
	then $w_{k_{j}}=\sigma_{k_{j-1}}$.
	Since $\sigma_{k_{j-1}}$ is a non-$g$-gap $h$-level excedance-letter of $\sigma$, we have
    $\sigma_{k_{j-1}}<k_{j-1}+g$ or $\sigma_{k_{j-1}}<h$.
    If $\sigma_{k_{j-1}}<h$, then $w_{k_{j}}=\sigma_{k_{j-1}}<h$,
    and the result holds.
    Below we assume that $\sigma_{k_{j-1}}\geq h$.
    Then we must have $\sigma_{k_{j-1}}<k_{j-1}+g$.
    It follows that $\sigma_{k_{j-1}}-g+1\leq k_{j-1}<k_{j}$.
    Thus, $k_{j-1}\in[\sigma_{k_{j-1}}-g+1,k_{j})=[w_{k_{j}}-g+1,k_{j})$.
    We claim that $w_{k_{j-1}}<k_{j-1}+g$.
    Note that $d+1\leq j-1\leq t$.
    Then $w_{k_{j-1}}$ is a non-$g$-gap $h$-level excedance-letter of $\sigma$.
    So $w_{k_{j-1}}<k_{j-1}+g$ or $w_{k_{j-1}}<h$.
    If $w_{k_{j-1}}<h$, 
    since we assumed $\sigma_{k_{j-1}}\geq h$,
    it follows that $w_{k_{j-1}}<h\leq \sigma_{k_{j-1}}<k_{j-1}+g$.
    Thus, in either case we have $w_{k_{j-1}}<k_{j-1}+g$,
    as claimed.
    By our claim we see that  $k_{j-1}\notin g\textsf{Exc}(w)$. 
    Combining this with $k_{j-1}\in[w_{k_{j}}-g+1,k_{j})$,
	we see that $w_{k_{j}}$ is not a critical non-$g$-gap excedance-letter of $w$, which is a contradiction.
\end{itemize}
Based on the above properties of $w$, 
we can uniquely recover $(\sigma,c)$ from its image $w$ as follows.
\begin{itemize}
	\item[$\bullet$]  If $z+g>n$,
	let $\sigma$ be the permutation obtained from $w$ by removing the letter $n$.
	
	\item[$\bullet$] If $\max\{z+g,h\}\leq a$, denote $e_{d}:=a$.
	Let $p=e_{d}-g+1$.
	Assume that $e_{j_{2}}<\cdots<e_{j_{x}}<e_{j_{x+1}}$ are the $g$-gap $h$-level excedance-letters of $w$ that lie strictly to the left of $w_{p}$ and are greater than $e_{d}$.
	Let $\sigma$ be obtained from $w$ by the following three steps:
	
	1. Delete the letter $e_{d}$ and replace it with a star ``$\ast$'';
	
	2. Shift the non-$g$-gap $h$-level excedance-letters to the right of the star one non-$g$-gap $h$-level excedance-letter position to the left (we think of the position occupied by the star as a non-$g$-gap $h$-level excedance-letter position);
	
	3. Replace $e_{j_{i}}$ with $e_{j_{i-1}}$ for all $i=2,3,\ldots,x+1$, where we set that $e_{j_{1}}=e_{d}$.
	
	\item[$\bullet$] If $a<\max\{z+g,h\}\leq n$,
	let $\sigma$ be  obtained from $w$ by the following two steps:
	
	1. Delete the letter $n$ and replace it with a star ``$\ast$'';

	2. Shift the non-$g$-gap $h$-level excedance-letters to the right of the star one non-$g$-gap $h$-level excedance-letter position to the left (we think of the position occupied by the star as a non-$g$-gap $h$-level excedance-letter position).
\end{itemize}
Now we have recovered $\sigma$ from $w$.
Finally, let $c=g\textsf{den}_{h}(w)-g\textsf{den}_{h}(\sigma)$,
and we successfully recover $(\sigma,c)$ from $w$.
Thus, the map $\phi_{g,h,n}^{\text{den}}: \mathfrak{S}_{n-1}\times\{0,1,\ldots,n-1\}\rightarrow\mathfrak{S}_{n}
$
is an injection.
Since $|\mathfrak{S}_{n-1}\times\{0,1,\ldots,n-1\}|=|\mathfrak{S}_{n}|=n!$,
we see that $\phi_{g,h,n}^{\text{den}}$ is a bijection.
\end{proof}

In the remainder of this section,  
we provide a proof of Theorem \ref{mian-general-theorem}.
To this end, we first state a lemma.
\begin{lemma}\label{Lemma_Exc_Exclp}
Let $g,\ell,h\geq1$ with $h\leq g+\ell$, and let $\pi$ be a permutation.
Then we have
$$g\emph{\textsf{Exc}}_{\ell}(\pi)=\{i\in g\emph{\textsf{Exclp}}_{h}(\pi):i\geq \ell\}.$$
\end{lemma}	
\begin{proof}
		Recall that
\vspace{-6pt}
	\begin{align*}
		g\textsf{Exc}_{\ell}(\pi)&=\{i:\pi_{i}\geq i+g,\!~i\geq \ell\},\text{~and~}\\
		g\textsf{Exclp}_{h}(\pi)&=\{i:\pi_{i}\geq i+g,\!~\pi_{i}\geq h\}.
	\end{align*}

\vspace{-10pt}
\noindent  If $i\in g\textsf{Exc}_{\ell}(\pi)$, we have $\pi_{i}\geq i+g\geq \ell+g\geq h$. 
Therefore, 
\begin{align*} 
	g\textsf{Exc}_{\ell}(\pi)=\{i:\pi_{i}\geq i+g,\!~\pi_{i}\geq h,\!~i\geq \ell\}=\{i\in g\textsf{Exclp}_{h}(\pi):i\geq \ell\},
\end{align*}	
completing the proof.
\end{proof}
\begin{proof}[Proof of Theorem \ref{mian-general-theorem}]
	If $n\leq g+\ell-1$, 
	then $g\textsf{exc}_{\ell}(\sigma)=|\{i:\sigma_{i}\geq i+g,\!~i\geq \ell\}|=0$ for any $\sigma\in\mathfrak{S}_{n}$.
	It is known that $g\textsf{den}_{h}$ is Mahonian; see \cite[the paragraph preceding Section 3]{Liu-2024}.
	Therefore, $(g\textsf{exc}_{\ell},g\textsf{den}_{h})$ is equidistributed with $(\textsf{0},\textsf{inv})$ over $\mathfrak{S}_{n}$.
	In what follows, we assume that $n>g+\ell-1$.
	Clearly, $h\leq g+\ell\leq n$.
	Let $\sigma\in\mathfrak{S}_{n-1}$.	
	Denote 
    \begin{align*}
	A_{g,h,\ell}(\sigma)=\{i\in g\textsf{Exclp}_{h}(\sigma):i\leq \ell-1\}
	\text{~and~}
	B_{g,h,\ell}(\sigma)=[\ell-1]-A_{g,h,\ell}(\sigma).
    \end{align*}
	Let $w=\phi_{g,h,n}^{\text{den}}(\sigma,c)$. Denote $s=|g\textsf{Exclp}_{h}(\sigma)|$.
	By Theorem \ref{Thm-bijection-denl}, we have
	\begin{align*}
		g\textsf{Exclp}_{h}(w)&=
		\left\{
		\begin{aligned}
			&g\textsf{Exclp}_{h}(\sigma), \quad\quad\quad\quad~\text{if~}0\leq c\leq s+g-1,\\
			&g\textsf{Exclp}_{h}(\sigma)\cup\{k_{d}\},\quad\text{~otherwise}
		\end{aligned}
		\right. \\
		g\textsf{den}_{h}(w)&=g\textsf{den}_{h}(\sigma)+c, 
	\end{align*}
	where  $d=c-s-g+1$ and $k_{d}$ is the $d$-th non-$g$-gap $h$-level excedance-letter position of $\sigma$ from smallest to largest. Then
	\begin{itemize}
		\item[$\bullet$] If $0\leq c\leq s+g-1$, we have
		$g\textsf{Exclp}_{h}(w)=g\textsf{Exclp}_{h}(\sigma)$.
		By Lemma \ref{Lemma_Exc_Exclp}, we see that $g\textsf{Exc}_{\ell}(w)=g\textsf{Exc}_{\ell}(\sigma)$.
		
		\item[$\bullet$] If $s+g\leq c\leq s+g+|B_{g,h,\ell}(\sigma)|-1$,
		we have $g\textsf{Exclp}_{h}(w)=g\textsf{Exclp}_{h}(\sigma)\cup\{k_{d}\}$.
		Since $d=c-s-g+1$, then $1\leq d\leq |B_{g,h,\ell}(\sigma)|$.
		It follows that $k_{d}\in B_{g,h,\ell}(\sigma)$, and so $k_{d}\leq \ell-1$.
		By Lemma \ref{Lemma_Exc_Exclp}, we have $g\textsf{Exc}_{\ell}(w)=g\textsf{Exc}_{\ell}(\sigma)$.
		
		\item[$\bullet$]
		If $c\geq s+g+|B_{g,h,\ell}(\sigma)|$,
		we have $g\textsf{Exclp}_{h}(w)=g\textsf{Exclp}_{h}(\sigma)\cup\{k_{d}\}$.
		Since $d=c-s-g+1$, then $d\geq|B_{g,h,\ell}(\sigma)|+1$.		
		It follows that $k_{d}\notin B_{g,h,\ell}(\sigma)$, and hence $k_{d}\geq\ell$. 
		By Lemma \ref{Lemma_Exc_Exclp}, we see that $g\textsf{Exc}_{\ell}(w)=g\textsf{Exc}_{\ell}(\sigma)\cup\{k_{d}\}$.
	\end{itemize}
	Therefore, in summary,
\begin{equation}\label{eq-(excr,denl)-all-in-summary}
	\begin{split}
		g\textsf{Exc}_{\ell}(w)&=\left
		\{
		\begin{aligned}
			&g\textsf{Exc}_{\ell}(\sigma), \quad\quad\quad\quad~\text{if~}0\leq c\leq s+g+|B_{g,h,\ell}(\sigma)|-1\\
			&g\textsf{Exc}_{\ell}(\sigma)\cup\{k_{d}\},\quad\text{~otherwise,}
		\end{aligned}
		\right.\\ 
		g\textsf{den}_{h}(w)&=g\textsf{den}_{h}(\sigma)+c.
	\end{split}
\end{equation}	
By Lemma \ref{Lemma_Exc_Exclp}, we see that $g\textsf{Exclp}_{h}(\sigma)$ is the disjoint union of $g\textsf{Exc}_{\ell}(\sigma)$ and $A_{g,h,\ell}(\sigma)$.
Thus, $s=g\textsf{exc}_{\ell}(\sigma)+|A_{g,h,\ell}(\sigma)|$.
It follows that     
\begin{align*}
	s+g+|B_{g,h,\ell}(\sigma)|-1=g\textsf{exc}_{\ell}(\sigma)+|A_{g,h,\ell}(\sigma)|+g+|B_{g,h,\ell}(\sigma)|-1
	=g\textsf{exc}_{\ell}(\sigma)+g+\ell-2.
\end{align*}
Combining this with (\ref{eq-(excr,denl)-all-in-summary}), we get
	\begin{equation*}
		\begin{split}
			g\textsf{exc}_{\ell}(w)&=\left
			\{
			\begin{aligned}
				&g\textsf{exc}_{\ell}(\sigma), \quad\quad\quad\text{if~}0\leq c\leq g\textsf{exc}_{\ell}(\sigma)+g+\ell-2,\\
				&g\textsf{exc}_{\ell}(\sigma)+1,\quad\text{~otherwise,}
			\end{aligned}
			\right.\\ 
			g\textsf{den}_{h}(w)&=g\textsf{den}_{h}(\sigma)+c.
		\end{split}
	\end{equation*}
	Comparing with equation (\ref{eq-r-Euler--Mahonian}),
	we see that $(g\textsf{exc}_{\ell},g\textsf{den}_{h})$ is $(g+\ell-1)$-Euler--Mahonian.
\end{proof}

\section*{Acknowledgment}
\addcontentsline{toc}{section}{Acknowledgment} 
This work was supported by the National Natural Science Foundation of China (12101134).

\phantomsection
\begin{spacing}{0.5}
	
\end{spacing}

\begin{thebibliography}{99}\setlength{\itemsep}{1.5mm}
		\addcontentsline{toc}{section}{References}
		\small{
			\bibitem{Carlitz-1975} L. Carlitz,
			A combinatorial property of $q$-Eulerian numbers,
			Amer. Math. Monthly 82 (1975) 51-54.
			
			\bibitem{Denert-1990}  M. Denert, 
			The genus zeta function of hereditary orders in central simple algebras over global fields,
			Math. Comput. 54 (1990) 449-465.
			
			\bibitem{Dumont-1974} D. Dumont,
			Interpr\'{e}tations Combinatoires des nombres de Genocchi, 
			Duke Math.J. 41 (1974) 305-317.
			
			\bibitem{Foata-1970} D. Foata, M.P. Sch\"{u}yzenberger,
			Th\'{e}orie G\'{e}om\'{e}trique des Polyn\^{o}mes Eul\'{e}riens,
			Lecture Notes in Mathematics No. 138, Springer-Verlag, Berlin, 1970.
			
			\bibitem{Foata-1990} D. Foata, D. Zeilberger,
			Denert's permutation statistic is indeed Euler--Mahonian,
			Stud. Appl. Math. 83 (1990) 31-59.
			
			\bibitem{Han-1990-direct}G.-N. Han,
			Une nouvelle bijection pour la statistique de Denert, 
			C. R. Acad. Sci. Paris, Ser. I 310 (1990) 493-496.
			
			\bibitem{Han-1991-thesis}G.-N. Han,
			Calcul Denertien, 
			Th\`{e}se de Doctorat, Publ. l'I.R.M.A., Strasbourg, 476/TS-29, 1991. 
			
			\bibitem{Yan-2025}K. Huang, Z. Lin, S.H.F. Yan,
			On a conjecture concerning the $r$-Euler--Mahonian statistic on permutations,
			J. Combin. Theory Ser. A 212 (2025) 106008.
			
			\bibitem{Yan-2025-2}K. Huang, S.H.F. Yan,
            Further results on $r$-Euler--Mahonian statistics,
            Adv. Appl. Math. 167 (2025) 102882.	
			
			\bibitem{Kasraoui-2009} A. Kasraoui,
			A classification of Mahonian maj-inv statistics,
			Adv. Appl. Math.  42(3) (2009) 342-357.
			
			\bibitem{Liu-2024} S.-H. Liu,
			$r$-Euler--Mahonian statistics on permutations,
			J. Combin. Theory Ser. A 208 (2024) 105940.
			
			\bibitem{MacMahon-1916}
			P.A. MacMahon, Combinatory Analysis, 2 volumes, Cambridge University Press, 1915.
			
			\bibitem{Magagnosc-1980} D. Magagnosc,
			Recurrences and formulae in an extension of the Eulerian numbers,
			Discrete Math. 30(3) (1980) 265-268.
			
			\bibitem{Rawlings-1981} D.P. Rawlings,
			The $r$-major index, J. Combin. Theory Ser. A 31(2) (1981) 175-183.
			
			\bibitem{Remmel-2015}J. Remmel, A.T. Wilson, 
			An extension of MacMahon's Equidistribution Theorem to ordered set partitions, 
			J. Combin. Theory Ser. A 134 (2015) 242-277.
			
			\bibitem{Riordan-1959} J. Riordan,
			An Introduction to Combinatorial Analysis, Wiley, New York, 1959.
			
			\bibitem{Stanley-2011} R.P. Stanley,  
			Enumerative combinatorics. Volume 1,
			Cambridge University Press, Cambridge, second edition, 2012.
		}
	\end{thebibliography}
\end{document}